\documentclass[final,leqno,onefignum,onetabnum]{siamltex1213}
\usepackage{enumerate}
\usepackage{amsmath}
\usepackage{amsfonts}
\usepackage{amssymb}
\usepackage{graphicx}
\usepackage{cite}%
\newtheorem{notation}[theorem]{Notation}
\newtheorem{remark}[theorem]{Remark}\title{Exact recursive estimation of linear systems subject to bounded
disturbances}

\author{Robin Hill, Yousong Luo and Uwe Schwerdtfeger\thanks{Robin Hill and Yousong Luo are with the School of Mathematical and
Geospatial Sciences, RMIT University, Latrobe St, Melbourne, 3001, Australia (\email{r.hill@rmit.edu.au, yluo@rmit.edu.au}), Uwe Schwerdtfeger is with the Department of Mathematics, Chemnitz University,
Germany (\email{uwe.schwerdtfeger@mathematik.tu-chemnitz.de})}}

\begin{document}
\maketitle
\slugger{sicon}{xxxx}{xx}{x}{x--x}

\begin{abstract}
This paper addresses the classical problem of determining the sets of possible
states of a linear discrete-time system subject to bounded disturbances from
measurements corrupted by bounded noise. These so-called uncertainty sets
evolve with time as new measurements become available. We present an exact,
computationally simple procedure that propagates a point on the boundary of
the uncertainty set at some time instant to a set of points on the boundary of
the uncertainty set at the next time instant.
\end{abstract}



\pagestyle{myheadings}
\thispagestyle{plain}
\markboth{Exact recursive estimation of linear systems subject to bounded
disturbances}{Hill, Luo and Schwerdtfeger}

\section{Introduction}

If a linear, time-invariant dynamic system is driven by set-bounded process
noise, and has measurements corrupted with set-bounded observation noise, then
the set of current possible states of the system consistent with the
measurements up to the current time is termed the \textit{state uncertainty
set }(or simply \textit{uncertainty set}), or sometimes the \textit{guaranteed
state estimate}. An algorithm for determining the uncertainty set is sometimes
called a \textit{set-valued observer}. This set membership estimation problem
is fundamental and has many applications, for example in control under
constraints in the presence of noise
\cite{Bertsekas-etal-1971,glover-schweppe-1971}. It falls under the general
topic of set membership uncertainty, see \cite{Combettes-93}. Recently there
has been interest in combining stochastic and set-bounded disturbances
\cite{Henningsson-2008}. The uncertainty set is needed in all of these
applications. Uncertainty set estimation is also closely related to
non-stochastic approaches to system identification
\cite{Fogel-Huang-82,milanese-etal-editors-1996,Ninness_Goodwin-95,Savkin_Petersen-98}.%

The first results on recursive determination of the uncertainty set are in
\cite{Schweppe-1968,Witsenhausen-minimax-1968,Witsenhausen-poss-states-1968}.
See also \cite{delfour-mitter-69}. Since the appearance of these papers there
has appeared an extensive literature on the topic; see the survey papers
\cite{Blanchini-2006,Milanese-Vicino-91}.

Most research to date has been on schemes that construct approximations to the
uncertainty set, for example
\cite{alamo-etal-2005,Blanchini-Sznaier-2012,elia-dahl,Savkin_Petersen-98,Shamma_Kuang-1999,Tempo-1988,Voulgaris-1995,Yang_Li-2011}%
. In the system identification literature there are results on exact recursive
polytope determination, for example \cite{Mo-Norton-1990}, where useful
descriptions of evolving polytopic uncertainty sets are given. We have the
same goal, but a completely different algorithm. Exact schemes generally have
not been suitable for real-time implementation because of their computational
complexity. In this paper we present for the first time a procedure that is
exact, recursive and computationally simple. When a new measurement arrives,
points on the boundary of the uncertainty set at the current time are mapped
exactly to points on the boundary of the uncertainty set at the next time
instant. The number of points that can be propagated forwards in time this way
is restricted only by speed and storage constraints, the computational
requirements for propagating one point being very small.

If the process and observation noise are restricted pointwise-in-time by
inequality constraints, then with the processing of more measurements the
number of vertices possessed by the polytopic uncertainty set may increase,
decrease, or stay the same. Each vertex of the uncertainty set at one time
instant may be mapped to either zero, one, or two vertices, or to an edge, of
the uncertainty set at the next time instant. Even if memory limitations
preclude the determination of all vertices, knowing the exact location of a
large number of points on the boundary of the uncertainty set potentially will
provide useful information in a wide range of applications. Exact
determination of the uncertainty set should also be of value in theoretical
work and in simulations.

There is a connection between uncertainty set estimation and research on
$l_{1}$ optimal control; \cite{Nagpal-poolla-1992,Stoorvogel-1996} provide
interesting insights on this. In the robustness literature problems with the
same number of disturbances as measurements, and the same number of controls
as regulated outputs, are referred to as one-block problems. See
\cite{hurak-zdenek-bottcher-2006} for a recent discussion of the one-block
$l_{1}$ optimal control problem. Our estimation problem has two disturbances,
one measurement and, because in this paper we are not attempting the next step
of using the estimate for closed-loop feedback, no controls or regulated
outputs. It is therefore a 2-block problem, where the disturbances are
connected by convolution constraints. As explained in \cite{staf-1993},
multi-block $l_{1}$ control problems necessarily have convolution constraints,
one-block problems have no convolution constraints, and so-called
zero-interpolation constraints, which ensure stability of the closed loop
system, may or may not be present in multi-block problems. If the measurements
in our estimation system are identically zero, the artificial regulator system
that we set up and recursively solve is a 2-block $l_{1}$ optimal control row
problem with no zero-interpolation conditions. When the measurements are
non-zero the cost function for the regulator system is no longer the $l_{1}-$
norm, but it remains piece-wise linear and convex. Thus the heart of our
procedure can be interpreted as recursively solving a slight generalization of
a 2-block $l_{1}$ optimal control problem. The results in this paper build on
some of the ideas in \cite{HILL-Luo-Schwerdtfeger-2012,HILL-CDC05,Hill-2008}.

Although there is no notion of optimality in the definition of uncertainty
sets, our procedure is derived using optimization methods. The uncertainty
sets are interpreted as feasible sets for specially constructed optimization
problems, and optimal solutions to these programs are points on the boundaries
of the uncertainty sets.

\section{Problem formulation}

A linear, time-invariant, causal discrete-time scalar system $\mathbf{z}%
=P\mathbf{v}+\mathbf{w}$ is depicted in Fig. \ref{basicfig}%
\begin{figure}[ptb]%
\centering
\includegraphics[scale=0.6]{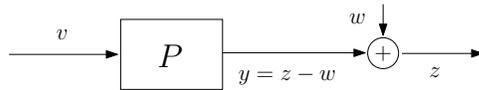}
\caption{Estimation system}%
\label{basicfig}%
\end{figure}
where $(v_{k})_{k=1}^{\infty},(z_{k})_{k=1}^{\infty}$ and $(w_{k}%
)_{k=1}^{\infty}$ are, respectively, the input disturbance to the plant $P$,
the measurement, and the measurement disturbance sequences. The plant output
sequence is $y_{k}=z_{k}-w_{k}.$ It is known \textit{a priori} that the
disturbances satisfy $\left\vert v_{k}\right\vert \leq1$ and $\left\vert
w_{k}\right\vert \leq1,$ and the initial state, at time $k=0,$ is given. The
restriction of $v_{k}$ and $w_{k}$ to the interval $\left[  -1,1\right]  $ is
made for notational convenience. The method to be described generalizes easily
to situations where $v_{k}$ is restricted to intervals of the form $\left[
v_{k}^{l},v_{k}^{u}\right]  ,$ and $w_{k}$ to $\left[  w_{k}^{l},w_{k}%
^{u}\right]  $, where $v^{l},v^{u},w^{l}$ and $w^{u}$ are \textit{a-priori}
given bounding sequences.

The state-space description best suited to our needs, given below, is related
to controllability form. The plant dynamics are also expressible, via the
transfer function representation of the system, as convolution constraints
relating $v$ and $y$; we shall make use of both of these system representations.

The problem addressed is: Given the \textit{a priori} information on
$w_{1},\ldots,w_{k},v_{1},\ldots,v_{k},$ the initial state $\mathbf{x}_{0},$
the measurement history $z_{1},\ldots,z_{k},$ and the plant dynamics, what are
the possible states at time $k,$ immediately after the measurement $z_{k}$ has
been received? The set of all such states, termed the uncertainty set at time
$k,$ will be denoted $S_{k},$ a convex polytope in $\mathbb{R}^{m},$ where $m$
is the order (McMillan degree) of the plant. Determining the set $S_{k}$ is an
estimation problem, and we shall refer to the system in Fig. \ref{basicfig} as
the estimation system.

\subsection{Notation and preliminaries\label{sectnotprelim}}

The boundary and interior of a set $S$ are denoted $\partial S$ and
$\operatorname*{int}S$, and $\emptyset$ denotes the empty set. The support
function of a convex, bounded non-empty subset $S$ of $\mathbb{R}^{m}$ is
$h_{S}(\mathbf{x}^{\ast})=\sup_{\mathbf{x}\in S}\left\langle \mathbf{x}^{\ast
},\mathbf{x}\right\rangle ,$ where $\mathbf{x}^{\ast}\in\mathbb{R}^{m}.$ The
cone $\left\{  \mathbf{x}^{\ast}:\left\langle \mathbf{x}^{\ast},\mathbf{x}%
\right\rangle =h_{S}(\mathbf{x}^{\ast}),\text{ }\mathbf{x}^{\ast}%
\neq\mathbf{0}\right\}  $ associated with $\mathbf{x}\in\partial S$ is denoted
$C_{S}^{O}\left(  \mathbf{x}\right)  .$ Given a vector $\mathbf{y=}\left(
y_{1},y_{2},\ldots\right)  $ and any $s\in\mathbb{N}^{+},$ $t\in\mathbb{N}%
^{+}$ satisfying $s<t,$ we denote $\left(  y_{s},y_{s+1},\ldots,y_{t}\right)
$ by $y_{s:t}.$ In matrix equations vectors are by default column vectors, so
for example $y_{s:t}$ occuring in a matrix equation would be a column vector,
and $y_{s:t}^{T}$\ is a row vector, where the superscript $T$ denotes
transpose. The vector of length $t+1$ whose first $t$ components are $y_{1:t}$
and whose last component is the scalar $y$ is denoted $\left(  y_{1:t}%
,y\right)  .$ The $\lambda$-transform (generating function) of an arbitrary
sequence $\mathbf{y}=(y_{k})_{k=1}^{\infty}$ is defined to be $\hat{y}%
(\lambda):=\sum_{k=1}^{\infty}y_{k}\lambda^{k-1}.$ Let $\mathbf{d}%
=d_{1:m+1}=(d_{1},\ldots,d_{m+1})$ and $\mathbf{n}=n_{1:m+1}=(n_{1}%
,\ldots,n_{m+1}),$ $m\geq1,$ be real vectors. The Toeplitz Bezoutian
$B_{T}(\mathbf{n},\mathbf{d})=\left(  b_{ij}\right)  _{i,j=1}^{m}$ of the
vectors $\mathbf{n},\mathbf{d}$ (or the polynomials $\hat{\mathbf{n}}%
,\hat{\mathbf{d}}$) is the $m\times m$ matrix with the generating polynomial
\begin{equation}
\sum_{i,j=1}^{m}b_{ij}t^{i-1}s^{j-1}=\frac{s^{m}\hat{\mathbf{n}}\left(
1/s\right)  \hat{\mathbf{d}}\left(  t\right)  -s^{m}\hat{\mathbf{d}}\left(
1/s\right)  \hat{\mathbf{n}}\left(  t\right)  }{1-st}. \label{def:BTdn}%
\end{equation}
Denote by $D$ and $N$ the infinite, banded, lower-triangular Toeplitz matrices
whose first columns are $\mathbf{d}$ and $\mathbf{n}$, respectively. Define
the following submatrices of $D$ and $N$
\[
D_{L}:=\left[
\begin{array}
[c]{cccc}%
d_{1} & 0 & \cdots & 0\\
d_{2} & d_{1} & \ddots & \vdots\\
\vdots & \ddots & \ddots & 0\\
d_{m} & \cdots & d_{2} & d_{1}%
\end{array}
\right]  \qquad D_{U}:=\left[
\begin{array}
[c]{cccc}%
d_{m+1} & d_{m} & \cdots & d_{2}\\
0 & d_{m+1} & \ddots & \vdots\\
\vdots & \ddots & \ddots & d_{m}\\
0 & \cdots & 0 & d_{m+1}%
\end{array}
\right]
\]%
\[
N_{L}:=\left[
\begin{array}
[c]{cccc}%
n_{1} & 0 & \cdots & 0\\
n_{2} & n_{1} & \ddots & \vdots\\
\vdots & \ddots & \ddots & 0\\
n_{m} & \cdots & n_{2} & n_{1}%
\end{array}
\right]  \qquad N_{U}:=\left[
\begin{array}
[c]{cccc}%
n_{m+1} & n_{m} & \cdots & n_{2}\\
0 & n_{m+1} & \ddots & \vdots\\
\vdots & \ddots & \ddots & n_{m}\\
0 & \cdots & 0 & n_{m+1}%
\end{array}
\right]  .
\]
More generally, for any $k>0,$ the $k\times k$ upper left hand corner
submatrix of $D$ is denoted $D_{k}.$ The matrix $N_{k}$ is defined similarly.
Thus $D_{m}=D_{L}$ and $N_{m}=N_{L}.$

One form of the Gohberg-Semencul formulas \cite{fuhrm,gohberg-semencul-1972} states%

\begin{equation}
B_{T}(\mathbf{n},\mathbf{d})=D_{L}N_{U}-N_{L}D_{U}=N_{U}D_{L}-D_{U}N_{L},
\label{def:BTmatrix}%
\end{equation}
and $B_{T}(\mathbf{n},\mathbf{d})$ is invertible if and only if $\hat
{\mathbf{n}}$ and $\hat{\mathbf{d}}$ are coprime. From now on we abbreviate
$B_{T}(\mathbf{n},\mathbf{d})$ to $B_{T},$ and $B_{T}^{-1}$ is the inverse of
$B_{T}.$ The first row of $B_{T}$ plays an important role and will be denoted
by $C,$ so $C:=\left(  d_{1}\left[
\begin{array}
[c]{c}%
n_{m+1}\\
\vdots\\
n_{2}%
\end{array}
\right]  -n_{1}\left[
\begin{array}
[c]{c}%
d_{m+1}\\
\vdots\\
d_{2}%
\end{array}
\right]  \right)  ^{T}$. See \cite{Heinig-Rost-2010} for properties of Bezoutians.

\subsection{Transfer function description}

The plant for the estimation system has the transfer function representation
$P(\lambda)=\hat{\mathbf{n}}(\lambda)/\hat{\mathbf{d}}(\lambda)$ where
\begin{align}
\hat{\mathbf{n}}(\lambda)  &  =n_{1}+n_{2}\lambda+n_{3}\lambda^{2}%
+\cdots+n_{m+1}\lambda^{m}\label{polynoms}\\
\hat{\mathbf{d}}(\lambda)  &  =1+d_{2}\lambda+d_{3}\lambda^{2}+\cdots
+d_{m+1}\lambda^{m},\nonumber
\end{align}
$m\geq1$ is an integer, $\hat{\mathbf{n}}(\lambda)$ and $\hat{\mathbf{d}%
}(\lambda)$ are assumed to be coprime polynomials with real coefficients, and
it is assumed that both the plant $P(\lambda)$ and the plant $P^{\ast}%
(\lambda)$ for the regulator system, defined below, are causal, implying
$d_{1}\neq0$ and $d_{m+1}\neq0$. Without loss of generality we take $d_{1}=1.$
Assuming zero initial conditions, $\mathbf{y}$ and $\mathbf{v}$ are related
by
\begin{equation}
\hat{\mathbf{d}}(\lambda)\hat{y}(\lambda)=\hat{\mathbf{n}}(\lambda)\hat
{v}(\lambda), \label{conv}%
\end{equation}
or equivalently $\mathbf{d}\ast\mathbf{y}=\mathbf{n}\ast\mathbf{v},$ where
$\ast$ denotes convolution.

Equating like powers of $\lambda$ on both sides of (\ref{conv}), and allowing
the possibility of non-zero initial conditions, we have
\begin{equation}
D\mathbf{y}-N\mathbf{v}=\left[
\begin{array}
[c]{c}%
D_{L}y_{1:m}-N_{L}v_{1:m}\\
0
\end{array}
\right]  . \label{conv_vy}%
\end{equation}
Equations (\ref{conv_vy}) describe how the signals $\mathbf{y}$ and
$\mathbf{v}$ are related in the estimation system. In the state-space
representation to be introduced next, $B_{T}^{-1}\left(  D_{L}y_{1:m}%
-N_{L}v_{1:m}\right)  $ is the initial state $\mathbf{x}_{0}.$

\subsection{State-space representations\label{sssection}}

The state-space description of the estimation system we employ is sometimes
denoted second controllability canonical form (\cite{LUEN-1979}, p 293). It
is
\begin{align}
\mathbf{x}_{k}  &  =A\mathbf{x}_{k-1}+Bv_{k}\nonumber\\
y_{k}  &  =C\mathbf{x}_{k-1}+D_{1}v_{k}\label{ss2}\\
z_{k}  &  =y_{k}+w_{k}\nonumber
\end{align}
where
\begin{subequations}\label{ABCD}
\begin{align}
A  &  =\left[
\begin{array}
[c]{cc}%
0 & I_{m-1}\\
-d_{m+1} &
\begin{array}
[c]{cc}%
\cdots & -d_{2}%
\end{array}
\end{array}
\right]  ,\text{ }B=\left[
\begin{array}
[c]{c}%
0\\
1
\end{array}
\right]  ,\\
C^{T}  &  =\left[
\begin{array}
[c]{c}%
n_{m+1}\\
\vdots\\
n_{2}%
\end{array}
\right]  -\left[
\begin{array}
[c]{c}%
d_{m+1}\\
\vdots\\
d_{2}%
\end{array}
\right]  n_{1},\text{ }D_{1}=n_{1};
\end{align}
\end{subequations}
and%
\begin{equation}
\mathbf{x}_{k}=\mathbf{x}_{k}(\mathbf{y},\mathbf{v}):=\left\{
\begin{array}
[c]{c}%
B_{T}^{-1}\left[  D_{L}y_{k+1:k+m}-N_{L}v_{k+1:k+m}\right]  \text{ for }%
k\geq0\\
B_{T}^{-1}\left[  -D_{U}y_{k-m+1:k}+N_{U}v_{k-m+1:k}\right]  \text{ for }k\geq
m
\end{array}
\right.  . \label{statedef}%
\end{equation}
In (\ref{ABCD}) $I_{m-1}$ denotes the $m-1$ dimensional identity matrix, and $0$ denotes a column vector of zeros of length $m-1$.
The fact that $D_{L}y_{k+1:k+m}-N_{L}v_{k+1:k+m}=-D_{U}y_{k-m+1:k}%
+N_{U}v_{k-m+1:k}$ for $k\geq m$ follows from (\ref{conv_vy}).

We will also require a state-space realization of a related system, which we
shall refer to as the regulator system. The input and output sequences are
respectively $(y_{k}^{\ast})_{k=1}^{\infty}$ and $(v_{k}^{\ast})_{k=1}%
^{\infty},$ and the plant regulator system, denoted $P^{\ast},$ has the
transfer function representation%

\begin{equation}
P^{\ast}(\lambda)=-\frac{\mathbf{\tilde{n}}(\lambda)}{\mathbf{\tilde{d}%
}(\lambda)} \label{regconv}%
\end{equation}
where $\mathbf{\tilde{n}}=\left(  n_{m+1},\ldots,n_{1}\right)  $ and
$\mathbf{\tilde{d}}=\left(  d_{m+1},\ldots,d_{1}\right)  .$ A minimal
state-space realization of the regulator system is%

\begin{align}
\mathbf{x}_{k}^{\ast}  &  =A^{\ast}\mathbf{x}_{k-1}^{\ast}+B^{\ast}y_{k}%
^{\ast}\label{adjointstate1}\\
v_{k}^{\ast}  &  =C^{\ast}\mathbf{x}_{k-1}^{\ast}+D_{1}^{\ast}y_{k}^{\ast}
\label{adjointstate}%
\end{align}%
\begin{align}
A^{\ast}  &  =\left[
\begin{array}
[c]{cc}%
-d_{m}/d_{m+1} & I_{m-1}\\%
\begin{array}
[c]{c}%
\vdots\\
-1/d_{m+1}%
\end{array}
&
\begin{array}
[c]{c}%
\\
0
\end{array}
\end{array}
\right]  ,\text{ }B^{\ast}=\left[
\begin{array}
[c]{c}%
n_{m}\\
\vdots\\
n_{1}%
\end{array}
\right]  -\left[
\begin{array}
[c]{c}%
d_{m}\\
\vdots\\
1
\end{array}
\right]  \frac{n_{m+1}}{d_{m+1}},\label{ABCDstardef}\\
C^{\ast}  &  =\left[
\begin{array}
[c]{cccc}%
-1/d_{m+1} & 0 & \cdots & 0
\end{array}
\right]  ,\text{ }D_{1}^{\ast}=\frac{-n_{m+1}}{d_{m+1}};\nonumber
\end{align}
and%
\begin{equation}
\mathbf{x}_{k}^{\ast}=\mathbf{x}_{k}^{\ast}(\mathbf{y}^{\ast},\mathbf{v}%
^{\ast}):=\left\{
\begin{array}
[c]{c}%
-N_{U}^{T}y_{k+1:k+m}^{\ast}-D_{U}^{T}v_{k+1:k+m}^{\ast}\text{ for }k\geq0\\
N_{L}^{T}y_{k-m+1:k}^{\ast}+D_{L}^{T}v_{k-m+1:k}^{\ast}\text{ for }k\geq m
\end{array}
\right.  , \label{state_def_star}%
\end{equation}
where $\mathbf{x}_{k}^{\ast}$ is the \textit{regulator state }at time
$k$\textit{. }From (\ref{regconv}) we have\textbf{ }$\mathbf{\tilde{n}}%
\ast\mathbf{y}^{\ast}+\mathbf{\tilde{d}}\ast\mathbf{v}^{\ast}=\left(  \left[
y_{1:m}^{\ast}\right]  ^{T}N_{U}+\left[  v_{1:m}^{\ast}\right]  ^{T}%
D_{U},0,\ldots\right)  ^{T}$ where the first component of the right hand
side vector is $-\mathbf{x}_{0}^{\ast}.$ These state-space representations are
in principle well known \cite{kailath_linsys,LUEN-1979,pold_willems}.

\subsection{The uncertainty set and worst-case disturbances}

The estimation system at time zero is in the state $\mathbf{x}_{0},$ so
$D_{L}y_{1:m}-N_{L}v_{1:m}=B_{T}\mathbf{x}_{0}.$ From the input-output
description (\ref{conv_vy}), after $k\geq2m$ measurements have been processed
$y_{1:k}$ and $v_{1:k}$ are related by%
\begin{equation}
\left[
\begin{array}
[c]{cccc}%
D_{L} &  &  & \\
D_{U} & D_{L} &  & \\
& \ddots & \ddots & \\
&  & D_{U} & D_{L}%
\end{array}
\right]  \left[
\begin{array}
[c]{c}%
y_{1}\\
\vdots\\
y_{k}%
\end{array}
\right]  -\left[
\begin{array}
[c]{cccc}%
N_{L} &  &  & \\
N_{U} & N_{L} &  & \\
& \ddots & \ddots & \\
&  & N_{U} & N_{L}%
\end{array}
\right]  \left[
\begin{array}
[c]{c}%
v_{1}\\
\vdots\\
v_{k}%
\end{array}
\right]  =\left[
\begin{array}
[c]{c}%
B_{T}\mathbf{x}_{0}\\
0\\
\vdots\\
0
\end{array}
\right]  , \label{conv_lengthk}%
\end{equation}
where here and later it is \textit{not} necessarily the case that $k$ be an
integer multiple of $m.$ In the notation of Section \ref{sectnotprelim},
(\ref{conv_lengthk}) is $D_{k}y_{1:k}-N_{k}v_{1:k}=\left[
\begin{array}
[c]{cccc}%
B_{T}\mathbf{x}_{0} & 0 & \cdots & 0
\end{array}
\right]  ^{T}.$ The uncertainty set $S_{k}$ is then given by%
\begin{equation}
S_{k}=\left\{
\begin{array}
[c]{c}%
\mathbf{x}\in\mathbb{R}^{m}:B_{T}\mathbf{x=}B_{T}\mathbf{x}_{k}(\mathbf{y}%
,\mathbf{v})=-D_{U}y_{k-m+1:k}+N_{U}v_{k-m+1:k},\\
\left\Vert v_{1:k}\right\Vert _{\infty}\leq1,\left\Vert y_{1:k}-z_{1:k}%
\right\Vert _{\infty}\leq1,\text{ and (\ref{conv_lengthk}) holds.}%
\end{array}
\right\}  \label{Sk1}%
\end{equation}

Following Witsenhausen, \cite{Witsenhausen-poss-states-1968}, $S_{k}$ is given
recursively in terms of $S_{k-1}$ and the new observation $z_{k}$ by
\begin{equation}
S_{k}=\left\{
\begin{array}
[c]{c}%
\mathbf{x}_{k}:\mathbf{x}_{k-1}\in S_{k-1},\text{ }\mathbf{x}_{k}%
=A\mathbf{x}_{k-1}+Bv_{k},\text{ }y_{k}=C\mathbf{x}_{k-1}+D_{1}v_{k},\\
\left\vert v_{k}\right\vert \leq1,\text{ }\left\vert y_{k}-z_{k}\right\vert
\leq1.
\end{array}
\right\}  \label{wits_recurs}%
\end{equation}

For states $\mathbf{x}_{k-1}$ and $\mathbf{x}_{k}$ related as in
(\ref{wits_recurs}) we shall say that $\mathbf{x}_{k-1}$ is a
\textit{precursor} of $\mathbf{x}_{k}$, and $\mathbf{x}_{k}$ is a
\textit{successor} to $\mathbf{x}_{k-1}.$

\begin{definition}
\label{defprecursor}The state $\mathbf{x}_{k-1}\in S_{k-1}$ is said to be a
\textit{precursor of the state }$\mathbf{x}_{k}\in S_{k},$ $\mathbf{x}_{k-1}$
is propagated to $\mathbf{x}_{k},$ and $\mathbf{x}_{k}$ is a successor to
$\mathbf{x}_{k-1},$ if there exists a scalar $v$ satisfying $\left\vert
v\right\vert \leq1$ and $\left\vert C\mathbf{x}_{k-1}+D_{1}v-z_{k}\right\vert
\leq1$ for which $\mathbf{x}_{k}=A\mathbf{x}_{k-1}+Bv$.
\end{definition}

Clearly every $\mathbf{x}_{k}\in S_{k}$ is a successor to some $\mathbf{x}%
_{k-1}\in S_{k-1}.$ The following Proposition follows directly from the
preceding definitions.

\begin{proposition}
\label{prop2}The vector $\mathbf{x}_{k}$ is a successor to $\mathbf{x}_{k-1}$
if and only if there exists $(y_{1:k},v_{1:k})$ satisfying (\ref{conv_lengthk}%
), $\left\Vert v_{1:k}\right\Vert _{\infty}\leq1,$ $\left\Vert y_{1:k}%
-z_{1:k}\right\Vert _{\infty}\leq1,$ $\mathbf{x}_{k}=\mathbf{x}_{k}\left(
y_{1:k},v_{1:k}\right)  $ and $\mathbf{x}_{k-1}=\mathbf{x}_{k-1}\left(
y_{1:k},v_{1:k}\right)  .$
\end{proposition}

At time $k$ any state $\mathbf{x}_{k}\in S_{k}$ is associated with possibly
non-unique disturbance histories $v_{1:k}$ and $w_{1:k}.$ Specifying one of
the disturbance histories uniquely determines the other, if the measurement
history $z_{1:k}$ and the initial state $\mathbf{x}_{0}$ are known. Thus the
state at time $k$ can be expressed in terms of the initial state and
$v_{1:k}.$ From (\ref{ss2}) we have
\begin{align}
\mathbf{x}_{k}  &  =A^{k}\mathbf{x}_{0}+%
{\displaystyle\sum\limits_{j=0}^{k-1}}
A^{j}Bv_{k-j}\label{state_meas}\\
w_{k}  &  =z_{k}-CA^{k-1}\mathbf{x}_{0}-C%
{\displaystyle\sum\limits_{j=0}^{k-2}}
A^{j}Bv_{k-j-1}-D_{1}v_{k}. \label{state_meas1}%
\end{align}

Every state $\mathbf{x}_{k}$ on the boundary of $S_{k}$ is determined, through
(\ref{state_meas}), by a so-called \textquotedblleft
worst-case\textquotedblright\ disturbance sequence $v_{1:k}$.

\begin{definition}
The signal $v_{1:k}$ is said to be a \textit{worst-case disturbance associated
with }$\mathbf{x}_{k}$ if $\mathbf{x}_{k}$ given by (\ref{state_meas})
satisfies $\mathbf{x}_{k}\in\partial S_{k}.$
\end{definition}

We will also say $\left(  w_{1:k},v_{1:k}\right)  $ are worst-case
disturbances associated with\textit{ }$\mathbf{x}_{k}$ if $v_{1:k}$ is a
worst-case disturbance associated with\textit{ }$\mathbf{x}_{k}$ and
(\ref{state_meas1}) holds.

In \cite{Witsenhausen-poss-states-1968} primal and dual recursive procedures
based on (\ref{wits_recurs}) are derived; they require manipulations of sets,
a computationally difficult task. Our recursion operates not on the whole set
$S_{k-1}$, but rather only on those boundary points of $S_{k-1}$ that are
precursors of boundary points of $S_{k}.$ We also apply the equation
$\mathbf{x}_{k}=A\mathbf{x}_{k-1}+Bv_{k},$ but only after first identifying
all suitable $v_{k}.$ By this is meant, for a given $\mathbf{x}_{k-1}\in
S_{k-1},$ finding all $v_{k}$ satisfying $\left\vert v_{k}\right\vert \leq1,$
$\left\vert y_{k}-z_{k}\right\vert \leq1$ having the property $A\mathbf{x}%
_{k-1}+Bv_{k}\in\partial S_{k}.$ Thus $\left(  z_{1:k}-y_{1:k},v_{1:k}\right)
$ are worst-case disturbances associated with $\mathbf{x}_{k},$ and
$\mathbf{x}_{k-1}$ is a precursor of $\mathbf{x}_{k}=A\mathbf{x}_{k-1}%
+Bv_{k}.$ The recursion we derive is exact and computationally simple. It is
novel in the uncertainty set membership literature in that primal and dual
recursions are intimately linked.

A description of the dual recursion is aided by some notation.

\begin{definition}
\label{defCone}Let $S$ be a polytope. The cone $\left\{  \mathbf{x}^{\ast
}:\left\langle \mathbf{x}^{\ast},\mathbf{x}\right\rangle =h_{S}(\mathbf{x}%
^{\ast}),\text{ }\mathbf{x}^{\ast}\neq\mathbf{0}\right\}  $ associated with
$\mathbf{x}\in\partial S$ is denoted $C_{S}^{O}\left(  \mathbf{x}\right)  .$
\end{definition}

Thus $C_{S}^{O}\left(  \mathbf{x}\right)  $ contains the directions of all
hyperplanes which touch $S$\ at $\mathbf{x.}$\ It is a basic result in the
theory polytopes that $C_{S}^{O}\left(  \mathbf{x}\right)  $ is non-empty.

While the primal recursion propagates $\mathbf{x}_{k-1}\in\partial S_{k-1}$ to
$\mathbf{x}_{k}\in\partial S_{k},$ the dual recursion propagates a regulator
state $\mathbf{x}_{k-1}^{\ast}\in C_{S_{k-1}}^{O}\left(  \mathbf{x}%
_{k-1}\right)  $ to $\mathbf{x}_{k}^{\ast}\in C_{S_{k}}^{O}\left(
\mathbf{x}_{k}\right)  .$ \linebreak See Fig. \ref{basic_conv}. The hyperplane with
normal $\mathbf{x}_{k-1}^{\ast}$ supports $S_{k-1}$ at $\mathbf{x}_{k-1}$.
Precursors $\mathbf{x}_{k-1}\in S_{k-1}$ of points $\mathbf{x}_{k}\in\partial
S_{k}$ for which $\mathbf{x}_{k-1}\in\partial S_{k-1}$ are most useful
because, as will be shown later, $S_{k}$ is the convex hull of the set
containing all propagations of all such $\mathbf{x}_{k-1}.$ There may also be
precursors $\mathbf{x}_{k-1}\in S_{k-1}$ of points $\mathbf{x}_{k}\in\partial
S_{k}$ for which $\mathbf{x}_{k-1}\notin\partial S_{k-1}.$ They will be
considered in Section \ref{sectinterior}.%

\begin{figure}[ptb]%
\centering
\includegraphics[scale=0.7]{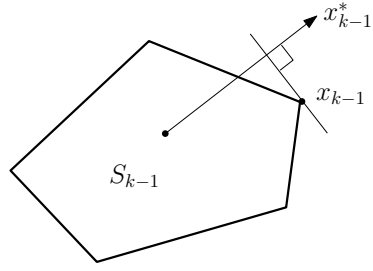}
\caption{The vectors $\mathbf{x}_{k-1}\in\partial S_{k-1}$ and $\mathbf{x}%
_{k-1}^{\ast}\in C_{S_{k-1}}^{O}(\mathbf{x}_{k-1})$ are propagated to
$\mathbf{x}_{k}\in\partial S_{k}$ and $\mathbf{x}_{k}^{\ast}\in C_{S_{k}}%
^{O}(\mathbf{x}_{k})$ by the measurement $z_{k}.$}%
\label{basic_conv}%
\end{figure}

\section{Statement of the procedure for propagating states\label{sect_recproc}%
}

From now on we will assume $\operatorname*{int}S_{k-1}\neq\emptyset$. In order
to state the procedure for propagating points $\mathbf{x}_{k-1}\in\partial
S_{k-1}$ we first introduce some definitions.


\begin{definition}
\label{def_align}The scalar pair $\left(  y,v\right)  $ is said to be aligned
at time $k$ with the scalar pair $\left(  y^{\ast},v^{\ast}\right)  $ if
\begin{equation}%
\begin{array}
[c]{c}%
v^{\ast}>0\Rightarrow v=1\\
v^{\ast}<0\Rightarrow v=-1\\
\left\vert v\right\vert <1\Rightarrow v^{\ast}=0
\end{array}
\label{align_defv}%
\end{equation}
and%
\begin{equation}%
\begin{array}
[c]{c}%
y^{\ast}>0\Rightarrow y=1+z_{k}\\
y^{\ast}<0\Rightarrow y=-1+z_{k}\\
\left\vert y-z_{k}\right\vert <1\Rightarrow y^{\ast}=0.
\end{array}
\label{align_defy}%
\end{equation}

\end{definition}

This definition can be extended in a natural way to alignment between pairs of
vector sequences. Thus the vector pair $\left(  y_{1:k},v_{1:k}\right)  $ is
aligned with the pair $\left(  y_{1:k}^{\ast},v_{1:k}^{\ast}\right)  $ if, for
all $j$, $\left(  y_{j},v_{j}\right)  $ is aligned at time $j$ with $\left(
y_{j}^{\ast},v_{j}^{\ast}\right)  .$

Given three scalars, a set consisting of quadruples of scalars is now defined.
It will play a central role.

\begin{definition}
\label{defM}Given scalars $s,$ $t$ and $z_{k},$ the set $M$ is%
\[
M\left(  s,t,z_{k}\right)  :=\left\{
\begin{array}
[c]{l}%
\mathbf{q}=\left(  v,y,v^{\ast},y^{\ast}\right)  \text{ satisfying}\\
1.\text{ \ }\left\vert v\right\vert \leq1,\left\vert y-z_{k}\right\vert
\leq1;\\
2.\text{ \ }y-n_{1}v=s;\\
3.\text{ \ }d_{m+1}v^{\ast}+n_{m+1}y^{\ast}=-t;\text{ and}\\
4.\text{ }\left(  y,v\right)  \text{ is aligned at time }k\text{ with }\left(
y^{\ast},v^{\ast}\right)  .
\end{array}
\right\}
\]

\end{definition}

The following Theorem, to be proved in Section \ref{combrecSect}, shows the
basic recursive idea, and the significance of the set $M.$

\begin{theorem}
\label{thmsuccessor}Suppose $\mathbf{x}_{k-1}\in\partial S_{k-1}$ and
$\mathbf{x}_{k-1}^{\ast}\in C_{S_{k-1}}^{O}(\mathbf{x}_{k-1})$. If \linebreak $\left(
v_{k},y_{k},v_{k}^{\ast},y_{k}^{\ast}\right)  \in M\left(  C\mathbf{x}%
_{k-1},\left(  \mathbf{x}_{k-1}^{\ast}\right)  _{1},z_{k}\right)  $ and
$\mathbf{x}_{k}^{\ast}:=A^{\ast}\mathbf{x}_{k-1}^{\ast}+B^{\ast}y_{k}^{\ast
}\neq\mathbf{0}$, then $\mathbf{x}_{k}:=A\mathbf{x}_{k-1}+Bv_{k}\in\partial
S_{k}$, $\mathbf{x}_{k}$ is a successor to $\mathbf{x}_{k-1}$, and
$\mathbf{x}_{k}^{\ast}\in C_{S_{k}}^{O}(\mathbf{x}_{k}).$%

\begin{figure}[ptb]%
\centering
\includegraphics[scale=0.54]{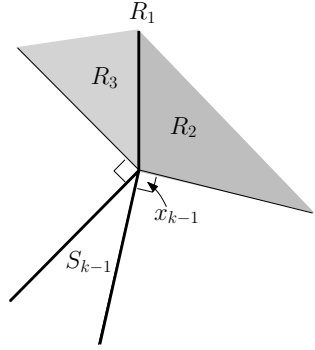}
\caption{The cone $C_{S_{k-1}}^{O}\left(  \mathbf{x}_{k-1}\right)  =R_{1}\cup
R_{2}\cup R_{3}$}%
\label{setsr}%
\end{figure}

\end{theorem}

Theorem \ref{thmsuccessor} can be used to find states on the boundary $S_{k}$,
but gives no guarantee of finding all states on the boundary of $S_{k}$. In
order to state results directed towards this goal, we need some more
definitions. The cone $C_{S_{k-1}}^{O}\left(  \mathbf{x}_{k-1}\right)  $
associated with any given $\mathbf{x}_{k-1}\in\partial S_{k-1}$ can be
partitioned into three disjoint cones:%
\begin{align*}
R_{1}  &  =R_{1}\left(  \mathbf{x}_{k-1}\right)  :=C_{S_{k-1}}^{O}\left(
\mathbf{x}_{k-1}\right)  \cap\left\{  \mathbf{x}_{k-1}^{\ast}:\left(
\mathbf{x}_{k-1}^{\ast}\right)  _{1}=0\right\} \\
R_{2}  &  =R_{2}\left(  \mathbf{x}_{k-1}\right)  :=C_{S_{k-1}}^{O}\left(
\mathbf{x}_{k-1}\right)  \cap\left\{  \mathbf{x}_{k-1}^{\ast}:\left(
\mathbf{x}_{k-1}^{\ast}\right)  _{1}>0\right\} \\
R_{3}  &  =R_{3}\left(  \mathbf{x}_{k-1}\right)  :=C_{S_{k-1}}^{O}\left(
\mathbf{x}_{k-1}\right)  \cap\left\{  \mathbf{x}_{k-1}^{\ast}:\left(
\mathbf{x}_{k-1}^{\ast}\right)  _{1}<0\right\}  .
\end{align*}

At least one of the $R_{i}$ is non-empty. See Fig. \ref{setsr}. One of the
$R_{i}$ is selected according to the following rule.%
\[
R=R\left(  \mathbf{x}_{k-1}\right)  :=\left\{
\begin{array}
[c]{l}%
R_{1}\text{ if }R_{1}\neq\emptyset\\
R_{2}\text{ if }R_{1}=\emptyset\text{ and }R_{2}\neq\emptyset\\
R_{3}\text{ if }R_{1}=\emptyset\text{ and }R_{3}\neq\emptyset\text{.}%
\end{array}
\right.
\]

This Definition makes sense because, if $R_{1}$ is empty, then precisely one
of $R_{2}$ and $R_{3}$ must be non-empty.

Given $\mathbf{x}_{k-1}\in\partial S_{k-1}$, any vector $\mathbf{x}%
_{k-1}^{\ast}\in R\left(  \mathbf{x}_{k-1}\right)  $, and $z_{k}$, we define
the sets $T$ and $X$.

\begin{definition}
\label{defti}%
\begin{align*}
T\left(  \mathbf{x}_{k-1},\mathbf{x}_{k-1}^{\ast},z_{k}\right)   &  :=\left\{
\begin{array}
[c]{c}%
\left(  \mathbf{x}_{k},\mathbf{x}_{k}^{\ast}\right)  =\left(  A\mathbf{x}%
_{k-1}+Bq_{1},A^{\ast}\mathbf{x}_{k-1}^{\ast}+B^{\ast}q_{4}\right)
\text{\textbf{ }satisfying}\\
\mathbf{x}_{k}^{\ast}\neq\mathbf{0}\text{ and }\mathbf{q}\in M\left(
C\mathbf{x}_{k-1},\left(  \mathbf{x}_{k-1}^{\ast}\right)  _{1},z_{k}\right)
\text{.}%
\end{array}
\right\} \\
\text{and }X  &  =X\left(  T\right)  :=\left\{  \mathbf{x}_{k}:\left(
\mathbf{x}_{k},\mathbf{x}_{k}^{\ast}\right)  \in T\right\}  \text{.}%
\end{align*}

\end{definition}

The set $T=T\left(  \mathbf{x}_{k-1},\mathbf{x}_{k-1}^{\ast},z_{k}\right)  $
can be empty. A useful observation is that although $T$ depends on the choice
of $\mathbf{x}_{k-1}^{\ast}\in R$, $X$ does not.

\begin{proposition}
\label{propTi}For any $\mathbf{x}_{k-1}\in\partial S_{k-1}$ and any $z_{k}$,
the set $X\left(  T\right)  $ does not depend on the choice of $\mathbf{x}%
_{k-1}^{\ast}\in R\left(  \mathbf{x}_{k-1}\right)  .$
\end{proposition}

\begin{proof}
Precisely one of $R=R_{1},$ $R=R_{2}$ or $R=R_{3}$ must hold. We show details
for the case $R=R_{2}.$ Select any $\mathbf{\bar{x}}_{k-1}^{\ast}\in R_{2}.$
The key observation is that, for any $\mathbf{q}\in M\left(  C\mathbf{x}%
_{k-1},\left(  \mathbf{\bar{x}}_{k-1}^{\ast}\right)  _{1},z_{k}\right)  $, it
is only the \textit{signs} of $q_{3}(=v^{\ast})$ and $q_{4}(=y^{\ast})$ that
restrict $q_{1}(=v)$ and $q_{2}(=y)$; that is, for the nine constraint
conditions in (\ref{align_defv}), (\ref{align_defy}) and Definition
\ref{defM}, the magnitudes of $v^{\ast}$ and $y^{\ast}$ do not constrain $v$
or $y$.\ But the possible signs of $v^{\ast}$ and $y^{\ast}$ satisfying
$d_{m+1}v^{\ast}+n_{m+1}y^{\ast}=-\left(  \mathbf{x}_{k-1}^{\ast}\right)
_{1}$ are the same for all $\mathbf{x}_{k-1}^{\ast}\in R_{2}$, because
$\left(  \mathbf{x}_{k-1}^{\ast}\right)  _{1}>0$ for all $\mathbf{x}%
_{k-1}^{\ast}\in R_{2}$. Thus $\left(  A\mathbf{x}_{k-1}+Bq_{1},A^{\ast
}\mathbf{\bar{x}}_{k-1}^{\ast}+B^{\ast}q_{4}\right)  \in T\left(
\mathbf{x}_{k-1},\mathbf{\bar{x}}_{k-1}^{\ast},z_{k}\right)  $ implies
$\left(  A\mathbf{x}_{k-1}+Bq_{1},A^{\ast}\mathbf{x}_{k-1}^{\ast}+B^{\ast
}q_{4}\right)  \in T\left(  \mathbf{x}_{k-1},\mathbf{x}_{k-1}^{\ast}%
,z_{k}\right)  $ for all $\mathbf{x}_{k-1}^{\ast}\in R_{2}$. The same argument
applies for the case $R=R_{3}$, and the case $R=R_{1}$ is similar.
\end{proof}

In light of this result, we write $X=X\left(  \mathbf{x}_{k-1},z_{k}\right)  $.

The main results of the paper are now presented. The ultimate aim is to
construct $S_{k}$, and this is achieved when the vertices of $S_{k}$ are
known. The following two Theorems provide the basis of a recursive procedure
for determining $\partial S_{k}$ from $\partial S_{k-1}$. Theorem
\ref{mainthm2} follows from Theorem \ref{thmsuccessor} and is proved in
Section \ref{combrecSect}. Theorem \ref{thmbig}, which is proved in Section
\ref{Sectend}, guarantees that vertices of $S_{k}$ have at least one precursor
on the boundary of $S_{k-1}$, and shows how any such precursor $\mathbf{x}%
_{k-1}$, and any $\mathbf{x}_{k-1}^{\ast}\in R\left(  \mathbf{x}_{k-1}\right)
$, are propagated.

\begin{theorem}
\label{mainthm2}Suppose $\mathbf{x}_{k-1}\in\partial S_{k-1}$. If
$\mathbf{x}_{k}\in X\left(  \mathbf{x}_{k-1},z_{k}\right)  $ then
$\mathbf{x}_{k}$ is a successor to $\mathbf{x}_{k-1}$ and $\mathbf{x}_{k}%
\in\partial S_{k}$.
\end{theorem}

\begin{theorem}
\label{thmbig}Let $\mathbf{x}_{k}$ be a vertex of $S_{k}$. Then there exists
$\mathbf{x}_{k-1}\in\partial S_{k-1}$ such that $\mathbf{x}_{k}\in X\left(
\mathbf{x}_{k-1},z_{k}\right)  $. Furthermore, for all $\mathbf{x}_{k-1}%
^{\ast}\in R\left(  \mathbf{x}_{k-1}\right)  $, there holds $\left(
\mathbf{x}_{k},\mathbf{x}_{k}^{\ast}\right)  \in T\left(  \mathbf{x}%
_{k-1},\mathbf{x}_{k-1}^{\ast},z_{k}\right)  $, where $\mathbf{x}_{k}^{\ast
}\in C_{S_{k}}^{O}\left(  \mathbf{x}_{k}\right)  $.
\end{theorem}

See Fig. \ref{backevmod} for a graphical illustration of finding $M\left(
C\mathbf{x}_{k-1},\left(  \mathbf{x}_{k-1}^{\ast}\right)  _{1},z_{k}\right)  $
and $T$. It depicts an Example where $\left(  \mathbf{x}_{k-1}^{\ast}\right)
_{1}<0$ for all $\mathbf{x}_{k-1}^{\ast}\in C_{S_{k-1}}^{O}\left(
\mathbf{x}_{k-1}\right)  ,$ so $R_{1}$ and $R_{2}$ are empty, and $R=R_{3}.$
For this Example $T$ contains the singleton element $\left(  A\mathbf{x}%
_{k-1}+B,A^{\ast}\mathbf{x}_{k-1}^{\ast}\right)  $ and $X=\left\{
A\mathbf{x}_{k-1}+B\right\}  .$ By Theorem \ref{mainthm2}, if $\mathbf{x}%
_{k-1}\in\partial S_{k-1}$ then $\mathbf{x}_{k}=A\mathbf{x}_{k-1}+B\in\partial
S_{k}$.

Determining the set $M$ does not become any more computationally demanding as
$m$ increases. For any $m$ it involves simply finding intersections of
straight lines in the plane and checking alignment.

The sets $M$ and $T$ are fundamental. Their description is aided by some
notation for points and lines in the plane.

\begin{notation}
\label{morenotation} Associated with any state $\mathbf{x}_{k-1}\in S_{k-1}$ there is
the line $y-n_{1}v=C\mathbf{x}_{k-1}$ in the $(y,v)$ plane, denoted $L(\mathbf{x}_{k-1}).$ Denote by $Q$
the set of points on or inside the square with vertices
$(1+z_{k},1),(1+z_{k},-1),(-1+z_{k},-1)$ and $(-1+z_{k},1)$.
\end{notation}

Not every $\mathbf{x}_{k-1}\in S_{k-1}$ has a successor. Although determining
successors $\mathbf{x}_{k}$ on the boundary of $S_{k}$ is the ultimate goal,
it is useful to first dispose of the simpler question of determining when
$\mathbf{x}_{k-1}\in S_{k-1}$ has a successor anywhere in $S_{k}$.

\begin{proposition}
\label{successorprop}The state $\mathbf{x}_{k-1}\in S_{k-1}$ has a successor
$\mathbf{x}_{k}\in S_{k}$ if and only if%
\[
\left\vert C\mathbf{x}_{k-1}-z_{k}\right\vert \leq\left\vert n_{1}\right\vert
+1.
\]
Furthermore, the set of all successors to $\mathbf{x}_{k-1}$ is%
\[
\left\{  \mathbf{x}_{k}:\mathbf{x}_{k}=A\mathbf{x}_{k-1}+Bv,\text{ }(y,v)\in
Q\cap L\left(  \mathbf{x}_{k-1}\right)  \right\}  .
\]

\end{proposition}

\begin{proof}
By Definition \ref{defprecursor}, $\mathbf{x}_{k-1}$ has a successor if and
only if scalars $v$ and $y$ exist for which $\left\vert v\right\vert \leq1,$
$\left\vert y-z_{k}\right\vert \leq1$ and $y=C\mathbf{x}_{k-1}+n_{1}v,$ in
which case the successor is $\mathbf{x}_{k}=A\mathbf{x}_{k-1}+Bv.$ By
elementary geometry of the plane such scalars $v$ and $y$ exist if and only if
the line $L\left(  \mathbf{x}_{k-1}\right)  $ intersects $Q,$
and the Proposition statements follow easily.
\end{proof}

The proof of the next Proposition is similar.

\begin{proposition}
For any $\mathbf{x}_{k-1}\in\partial S_{k-1}$, and any $\mathbf{x}_{k-1}%
^{\ast}\in C_{S_{k-1}}^{O}\left(  \mathbf{x}_{k-1}\right)  $, the sets
\newline$M\left(  C\mathbf{x}_{k-1},\left(  \mathbf{x}_{k-1}^{\ast}\right)
_{1},z_{k}\right)  $, $T$ and $X$ are empty if $\left\vert C\mathbf{x}%
_{k-1}-z_{k}\right\vert >\left\vert n_{1}\right\vert +1.$
\end{proposition}

The following two Propositions follow easily from the obvious fact that the
line $L\left(  \mathbf{x}_{k-1}\right)  $ can intersect the
boundary of $Q$ at most twice. Let $\mathbf{x}_{k-1}\in\partial S_{k-1}$.

\begin{proposition}
If $C_{S_{k-1}}^{O}\left(  \mathbf{x}_{k-1}\right)  $ is such that
$R_{1}=\emptyset$ then the possible values of $\operatorname*{card}\left(
X\left(  \mathbf{x}_{k-1},z_{k}\right)  \right)  $ are $0,1$ and $2.$
\end{proposition}

\begin{proposition}
If $C_{S_{k-1}}^{O}\left(  \mathbf{x}_{k-1}\right)  $ is such that $R_{1}%
\neq\emptyset$ then the set $X\left(  \mathbf{x}_{k-1},z_{k}\right)  $ is
either empty, contains one element, or is the one-dimensional line
segment\newline$\left\{  \mathbf{x}_{k}:\mathbf{x}_{k}=A\mathbf{x}%
_{k-1}+Bv,\text{ }v\in\lbrack v_{\min},v_{\max}]\right\}  $ where $v_{\min}$
and $v_{\max}$ are the minimum and maximum values of $v$ for which the line
$L\left( \mathbf{x}_{k-1}\right)  $ intersects the sides of $Q.$
\end{proposition}

To proceed further we need duality. The proofs of the Theorems in this Section
are based on the duality existing between programs constructed from the
estimator and regulator systems.%

\begin{figure}[ptb]%
\centering
\includegraphics[scale=0.65]{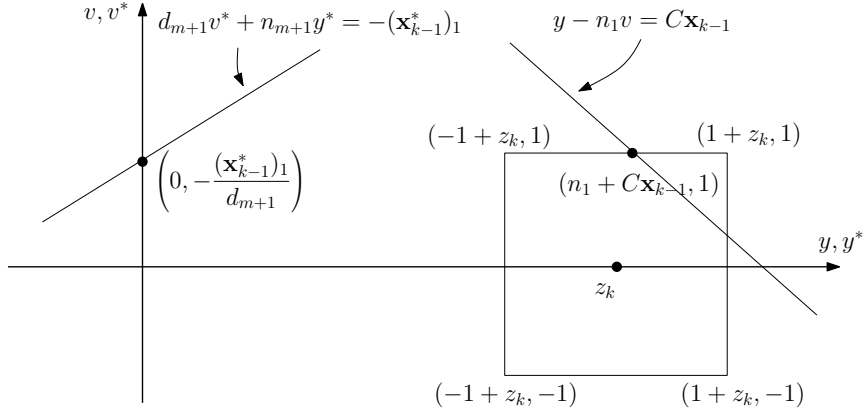}
\caption{For this measurement $z_{k},$ the location of the square implies that
the unique element in $M\left(  C\mathbf{x}_{k-1},\left(  \mathbf{x}%
_{k-1}^{\ast}\right)  _{1},z_{k}\right)  $ is $\left(  1,n_{1}+C\mathbf{x}%
_{k-1},\frac{-\left(  \mathbf{x}_{k-1}^{\ast}\right)  _{1}}{d_{m+1}},0\right)
$. }%
\label{backevmod}%
\end{figure}

\section{The estimator program $\mathcal{E}_{z_{1:k}}(\mathbf{x}^{\ast})$}

From now on we will always assume $k\geq2m$ and $S_{k}\neq\emptyset$. The
optimization problem we construct is based on the support function for the set
$S_{k}$. Since $S_{k}$ is compact its support function is $h_{S_{k}%
}(\mathbf{x}^{\ast})=\max_{\mathbf{x}\in S_{k}}\left\langle \mathbf{x}^{\ast
},\mathbf{x}\right\rangle $, and the hyperplane $\left\{  \mathbf{x}%
:\left\langle \mathbf{x}^{\ast},\mathbf{x}\right\rangle =h_{S_{k}}%
(\mathbf{x}^{\ast})\right\}  $ in the direction $\mathbf{x}^{\ast}$ supports
$S_{k}$ at $\mathbf{x}.$ For any $\mathbf{x}^{\ast}\in\mathbb{R}^{m},$ define
the \textit{estimator program}
\[
\mathcal{E}_{z_{1:k}}(\mathbf{x}^{\ast}):\text{ \ \ }\max\limits_{\mathbf{x}%
\in S_{k}}\left\langle \mathbf{x}^{\ast},\mathbf{x}\right\rangle .
\]
It has optimal value $h_{S_{k}}(\mathbf{x}^{\ast}).$ The notation
$\mathcal{E}_{z_{1:k}}(\mathbf{\cdot})$ will be used to denote the estimator
program when $\mathbf{x}^{\ast}$ is not important or not specified.

The following Proposition follows directly from the definitions.

\begin{proposition}
\label{propconeargmax} For any $\mathbf{x}\in S_{k}$ and any $\mathbf{x}%
^{\ast}\in\mathbb{R}^{m},$ there holds
\[
\mathbf{x}\in\arg\max\mathcal{E}_{z_{1:k}}(\mathbf{x}^{\ast})\Leftrightarrow
h_{S_{k}}(\mathbf{x}^{\ast})=\left\langle \mathbf{x}^{\ast},\mathbf{x}%
\right\rangle .
\]
\newline Furthermore, for any $\mathbf{x}\in\partial S_{k}$ and $\mathbf{0\neq
x}^{\ast}\in\mathbb{R}^{m}$ there holds%
\[
\mathbf{x}\in\arg\max\mathcal{E}_{z_{1:k}}(\mathbf{x}^{\ast})\Leftrightarrow
\mathbf{x}^{\ast}\in C_{S_{k}}^{O}(\mathbf{x}).
\]

\end{proposition}

If $S_{k}$ is non-empty and $\mathbf{x}^{\ast}\neq\mathbf{0}$ then optimizing
$\mathbf{x}$ must be on the boundary of $S_{k}$, and $\arg\max\mathcal{E}%
_{z_{1:k}}(\mathbf{x}_{k}^{\ast})$ is a non-empty subset of $\partial S_{k}.$
Any point in $\partial S_{k}$ will belong to $\arg\max\mathcal{E}_{z_{1:k}%
}(\mathbf{x}^{\ast})$ for some $\mathbf{x}^{\ast}\neq\mathbf{0}$. All of these
statements are simple consequences of $S_{k}$ being convex and compact. Some
elementary properties relating optimal solutions to the program $\mathcal{E}%
_{z_{1:k}}(\mathbf{x}^{\ast})$ with geometry of the polytope $S_{k}$ are
collected in the next Proposition.

\begin{proposition}
\label{basicprop}Suppose $S_{k}$ is non-empty. Then\newline1) if
$\mathbf{x}\in\partial S_{k}$ and $\mathbf{x}^{\ast}\neq\mathbf{0}$ is the
direction of any hyperplane supporting $S_{k}$ at $\mathbf{x}$, then
$\mathbf{x}\in\arg\max\mathcal{E}_{z_{1:k}}(\mathbf{x}^{\ast})$;\newline2) if
$\mathbf{x}\in\partial S_{k}$ then there exists $\mathbf{x}^{\ast}%
\neq\mathbf{0}$ for which $\mathbf{x}\in\arg\max\mathcal{E}_{z_{1:k}%
}(\mathbf{x}^{\ast})$;\newline3) if $\mathbf{x}\in\arg\max\mathcal{E}%
_{z_{1:k}}(\mathbf{x}^{\ast})$ and $\mathbf{x}^{\ast}\neq\mathbf{0}$\textbf{,}
then $\mathbf{x}^{\ast}$ is the direction of a hyperplane supporting $S_{k}$
at $\mathbf{x}$;\newline4) if $\mathbf{x}\in\arg\max\mathcal{E}_{z_{1:k}%
}(\mathbf{x}^{\ast})$ and $\mathbf{x}\in\operatorname*{int}S_{k}$, then
$\mathbf{x}^{\ast}=\mathbf{0}$\textbf{; }and\newline5) $\arg\max$
$\mathcal{E}_{z_{1:k}}(\mathbf{0})=S_{k}$.
\end{proposition}

A program almost identical to $\mathcal{E}_{z_{1:k}}(\mathbf{x}^{\ast})$,
denoted $\mathcal{E}_{z_{1:k}}^{\prime}(\mathbf{x}^{\ast})$, is introduced for
notational clarity. By (\ref{Sk1}) $\mathcal{E}_{z_{1:k}}(\mathbf{x}^{\ast})$
can be equivalently expressed as%
\[%
\begin{array}
[c]{c}%
\mathcal{E}_{z_{1:k}}^{\prime}(\mathbf{x}^{\ast}):\text{ \ \ }\max
\limits_{y_{1:k},v_{1:k}}\left\langle \mathbf{x}^{\ast},\mathbf{x}%
\right\rangle \\
\text{subject to }\\
\left\Vert y_{1:k}-z_{1:k}\right\Vert _{\infty}\leq1,\text{ }\left\Vert
v_{1:k}\right\Vert _{\infty}\leq1\text{ and}\\
\left[
\begin{array}
[c]{cccc}%
D_{L} &  &  & \\
D_{U} & D_{L} &  & \\
& \ddots & \ddots & \\
&  & D_{U} & D_{L}\\
&  &  & D_{U}%
\end{array}
\right]  \left[
\begin{array}
[c]{c}%
y_{1}\\
\vdots\\
y_{k}%
\end{array}
\right]  -\\
\left[
\begin{array}
[c]{cccc}%
N_{L} &  &  & \\
N_{U} & N_{L} &  & \\
& \ddots & \ddots & \\
&  & N_{U} & N_{L}\\
&  &  & N_{U}%
\end{array}
\right]  \left[
\begin{array}
[c]{c}%
v_{1}\\
\vdots\\
v_{k}%
\end{array}
\right]  =\left[
\begin{array}
[c]{c}%
B_{T}\mathbf{x}_{0}\\
0\\
\vdots\\
0\\
-B_{T}\mathbf{x}%
\end{array}
\right]
\begin{array}
[c]{c}%
\text{fixed}\\
\\
\\
\\
\text{free}%
\end{array}
\end{array}
.
\]
By (\ref{statedef}) the final state $\mathbf{x}$ satisfies $\mathbf{x=x}%
_{k}\left(  y_{1:k},v_{1:k}\right)  $. See Fig. \ref{Dfig}. The decision
variables for the program $\mathcal{E}_{z_{1:k}}^{\prime}(\mathbf{x}^{\ast})$
are the outputs and inputs of the estimation system up to time $k.$ From now
on we put $(\mathbf{y},\mathbf{v}):=\left(  y_{1:k},v_{1:k}\right)  ,$ and
later $(\mathbf{y}^{\ast},\mathbf{v}^{\ast}):=\left(  y_{1:k}^{\ast}%
,v_{1:k}^{\ast}\right)  $. The relationship between estimator signals
$(\mathbf{y},\mathbf{v})$ and states $\mathbf{x}\in\partial S_{k}$ occurring
as optimizing decision variables in the programs $\mathcal{E}_{z_{1:k}%
}(\mathbf{x}^{\ast})$ and $\mathcal{E}_{z_{1:k}}^{\prime}(\mathbf{x}^{\ast})$
is summarized in the following Proposition.

\begin{proposition}
\label{basicprop1} 1) For all $\mathbf{x}\in\partial S_{k}$, and for all
$\mathbf{x}^{\ast}\in C_{S_{k}}^{O}(\mathbf{x}),$ there exists $\left(
\mathbf{y},\mathbf{v}\right)  $ $\in\arg\max\mathcal{E}_{z_{1:k}}^{\prime
}(\mathbf{x}^{\ast})$ for which $\mathbf{x}=\mathbf{x}_{k}\left(
\mathbf{y},\mathbf{v}\right)  $.\newline2) Suppose $\mathbf{x}^{\ast}%
\in\mathbb{R}^{m}$ and $(\mathbf{y},\mathbf{v})\in\mathbb{R}^{k}%
\times\mathbb{R}^{k}$. Then $(\mathbf{y},\mathbf{v})\in\arg\max\mathcal{E}%
_{z_{1:k}}^{\prime}(\mathbf{x}^{\ast})$ if and only if $\mathbf{x}%
_{k}(\mathbf{y},\mathbf{v})\in\arg\max\mathcal{E}_{z_{1:k}}(\mathbf{x}^{\ast
})$.
\end{proposition}

Note also that the origin may or may not be in $S_{k}$. If $S_{k}$ does not
contain the origin then there will exist $\mathbf{x}^{\ast}$ for which
$h_{S_{k}}\left(  \mathbf{x}^{\ast}\right)  <0.$

\section{The regulator program $\mathcal{R}_{z_{1:k}}(\mathbf{x}^{\ast})$}

We would like to use a dynamic programming style argument to determine all
optimal solutions to the program $\mathcal{E}_{z_{1:k}}(\mathbf{x}_{k}^{\ast
})$ recursively from a known optimal solution to $\mathcal{E}_{z_{1:k-1}%
}(\mathbf{x}_{k-1}^{\ast}),$ where $\mathbf{x}_{k}^{\ast}$ is determined
recursively from $\mathbf{x}_{k-1}^{\ast}.$ Such a recursion would yield
point(s) on the boundary of the feasible set for $\mathcal{E}_{z_{1:k}%
}(\mathbf{x}_{k}^{\ast}),$ the desired points on the boundary of $S_{k}.$
However, the cost function for the program $\mathcal{E}_{z_{1:k}}%
(\mathbf{x}_{k}^{\ast})$ is not in a form suitable for application of dynamic
programming. We make use of a program with a dual pairing to $\mathcal{E}%
_{z_{1:k}}(\mathbf{x}_{k}^{\ast}),$ termed the \textit{regulator program}, and
denoted $\mathcal{R}_{z_{1:k}}(\mathbf{x}_{k}^{\ast}),$ for which the cost
function is of a suitable form. Although a straightforward application of
dynamic programming to $\mathcal{R}_{z_{1:k}}(\mathbf{x}_{k}^{\ast})$ by
itself does not yield a computationally tractable recursion, we show that
linking the optimal solutions to $\mathcal{R}_{z_{1:k}}(\mathbf{x}_{k}^{\ast
})$ and $\mathcal{E}_{z_{1:k}}(\mathbf{x}_{k}^{\ast})$ through alignment
(complementary slackness) conditions, in conjunction with the use of dynamic
programming, does yield the desired recursion.

The duality between $\mathcal{R}_{z_{1:k}}(\mathbf{x}^{\ast})$ and
$\mathcal{E}_{z_{1:k}}(\mathbf{x}^{\ast})$ will now be interpreted in the
structural form required to carry through this plan. The \textit{regulator
program} is defined as:%
\[%
\begin{array}
[c]{c}%
\mathcal{R}_{z_{1:k}}(\mathbf{x}^{\ast}):\text{ \ \ }\min\limits_{\mathbf{y}%
^{\ast},\mathbf{v}^{\ast}}\left\{  \left\Vert \mathbf{y}^{\ast}\right\Vert
_{1}+\left\Vert \mathbf{v}^{\ast}\right\Vert _{1}+\left\langle y_{1:k}^{\ast
},z_{1:k}\right\rangle +\left\langle \mathbf{x}_{0}^{\ast},\mathbf{x}%
_{0}\right\rangle \right\} \\
\text{subject to }\\
\left[
\begin{array}
[c]{cccc}%
N_{U}^{T} &  &  & \\
N_{L}^{T} & N_{U}^{T} &  & \\
& \ddots & \ddots & \\
&  & N_{L}^{T} & N_{U}^{T}\\
&  &  & N_{L}^{T}%
\end{array}
\right]  \left[
\begin{array}
[c]{c}%
y_{1}^{\ast}\\
\vdots\\
y_{k}^{\ast}%
\end{array}
\right]  +\\
\left[
\begin{array}
[c]{cccc}%
D_{U}^{T} &  &  & \\
D_{L}^{T} & D_{U}^{T} &  & \\
& \ddots & \ddots & \\
&  & D_{L}^{T} & D_{U}^{T}\\
&  &  & D_{L}^{T}%
\end{array}
\right]  \left[
\begin{array}
[c]{c}%
v_{1}^{\ast}\\
\vdots\\
v_{k}^{\ast}%
\end{array}
\right]  =\left[
\begin{array}
[c]{c}%
-\mathbf{x}_{0}^{\ast}\\
0\\
\vdots\\
0\\
\mathbf{x}^{\ast}%
\end{array}
\right]
\begin{array}
[c]{c}%
\text{free}\\
\\
\\
\\
\text{fixed}%
\end{array}
\end{array}
\]

The decision variables are the inputs and outputs of the regulator system up
to time $k$ described in Section \ref{sssection}. See Fig. \ref{Pfig}. If the
measurements are all zero, $\mathbf{x}_{0}=\mathbf{0}$ and $k\rightarrow
\infty$, then $\mathcal{R}_{z_{1:k}}(\mathbf{x}^{\ast})$ has an interpretation
as a time-reversed deterministic $l_{1}$-norm regulator problem, where the
input and output signals are made as small as possible and driven
asymptotically to zero.

A formal statement of the duality existing between $\mathcal{E}_{z_{1:k}%
}^{\prime}(\mathbf{x}^{\ast})$ and $\mathcal{R}_{z_{1:k}}(\mathbf{x}^{\ast})$
is now stated.

\begin{proposition}
\label{primdaulfixed}Suppose the set $S_{k}$ is non-empty. Then the optimal
values of $\mathcal{E}_{z_{1:k}}^{\prime}(\mathbf{x}^{\ast})$ and
$\mathcal{R}_{z_{1:k}}(\mathbf{x}^{\ast})$ are finite and equal. Furthermore,
if $(\mathbf{y},\mathbf{v})$ and $(\mathbf{y}^{\ast},\mathbf{v}^{\ast})$ are
feasible for $\mathcal{E}_{z_{1:k}}^{^{\prime}}(\mathbf{x}^{\ast})$ and
$\mathcal{R}_{z_{1:k}}(\mathbf{x}^{\ast}),$ respectively, then a necessary and
sufficient condition that they both be optimal is that they be
\textit{aligned.}
\end{proposition}

\begin{proof}
The proof in outline follows standard linear programming arguments. The
Gohberg-Semencul formula (\ref{def:BTmatrix}) is also required. Details are in
the Appendix.
\end{proof}

\begin{remark}
For the program $\mathcal{R}_{z_{1:k}}(\mathbf{x}^{\ast})$ the initial state
is free, and the terminal state is fixed, at $\mathbf{x}^{\ast}$. For the
program $\mathcal{E}_{z_{1:k}}^{\prime}(\mathbf{x}^{\ast})$ the initial state
is fixed, at $\mathbf{x}_{0}$, and the terminal state is free$.$
\end{remark}



%

\begin{figure}[ptb]%
\centering
\includegraphics[scale=0.6]{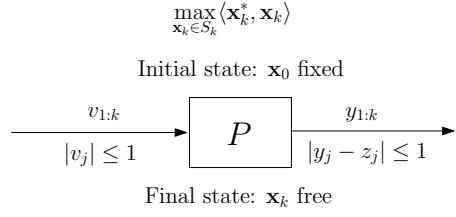}
\caption{The program $\mathcal{E}_{z_{1:k}}(\mathbf{x}_{k}^{\ast}).$ If
$\mathbf{x}_{k}^{\ast}\neq\mathbf{0}$ then optimizing $\mathbf{x}_{k}$ are
points on the boundary of $S_{k}.$}%
\label{Dfig}%
\end{figure}

%
\begin{figure}[ptb]%
\centering
\includegraphics[scale=0.6]{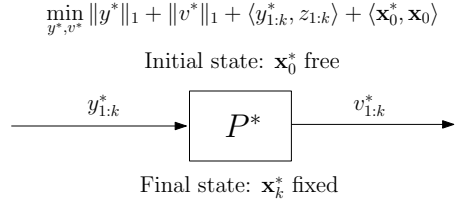}
\caption{The program $\mathcal{R}_{z_{1:k}}(\mathbf{x}_{k}^{\ast})$}%
\label{Pfig}%
\end{figure}


\section{Combined recursion in the estimator and regulator
programs\label{combrecSect}}

Our goal is determine when and how a state $\mathbf{x}_{k-1}\in\partial
S_{k-1}$ is propagated to a successor $\mathbf{x}_{k}\in\partial S_{k}$. Now
$\mathbf{x}_{k-1}$ has at least one associated worst-case disturbance
$v_{1:k-1},$ and if $\left\langle \mathbf{x}_{k-1}^{\ast},\mathbf{x}%
_{k-1}\right\rangle =h_{S_{k-1}}(\mathbf{x}_{k-1}^{\ast})$ then $\left(
y_{1:k-1},v_{1:k-1}\right)  \in\arg\max\mathcal{E}_{z_{1:k-1}}^{\prime
}(\mathbf{x}_{k-1}^{\ast})$ where, by (\ref{state_meas1}), $y_{1:k-1}$ is
uniquely determined by $v_{1:k-1},$ $z_{1:k-1}$ and $\mathbf{x}_{0}.$ By
Proposition \ref{primdaulfixed}, for all such $\left(  y_{1:k-1}%
,v_{1:k-1}\right)  \in\arg\max\mathcal{E}_{z_{1:k-1}}^{\prime}(\mathbf{x}%
_{k-1}^{\ast}),$ there exists $\left(  y_{1:k-1}^{\ast},v_{1:k-1}^{\ast
}\right)  \in\arg\min$$\mathcal{R}_{z_{1:k-1}}(\mathbf{x}_{k-1}^{\ast}),$ and
$\left(  y_{1:k-1}^{\ast},v_{1:k-1}^{\ast}\right)  $ is aligned with $\left(
y_{1:k-1},v_{1:k-1}\right)  .$ The next Proposition yields useful extensions
to $\left(  y_{1:k-1},v_{1:k-1}\right)  $ and $\left(  y_{1:k-1}^{\ast
},v_{1:k-1}^{\ast}\right)  $.

\begin{proposition}
\label{mainthm} Suppose $\mathbf{x}_{k-1}$ and $\mathbf{x}_{k-1}^{\ast}$
satisfying $\left\langle \mathbf{x}_{k-1},\mathbf{x}_{k-1}^{\ast}\right\rangle
=h_{S_{k-1}}(\mathbf{x}_{k-1}^{\ast})$ are given. Select any $\left(
y_{1:k-1},v_{1:k-1}\right)  \in\arg\max\mathcal{E}_{z_{1:k-1}}^{\prime
}(\mathbf{x}_{k-1}^{\ast}),$ and any \linebreak $\left(  y_{1:k-1}^{\ast},v_{1:k-1}%
^{\ast}\right)  \in\arg\min$$\mathcal{R}_{z_{1:k-1}}(\mathbf{x}_{k-1}^{\ast
}).$ Then $\left(  y_{1:k},v_{1:k}\right)  \in\arg\max\mathcal{E}_{z_{1:k}%
}^{\prime}(A^{\ast}\mathbf{x}_{k-1}^{\ast}+B^{\ast}y_{k}^{\ast})$ and $\left(
y_{1:k}^{\ast},v_{1:k}^{\ast}\right)  \in\arg\min$$\mathcal{R}_{z_{1:k}%
}(A^{\ast}\mathbf{x}_{k-1}^{\ast}+B^{\ast}y_{k}^{\ast})$ if $\left(
v_{k},y_{k},v_{k}^{\ast},y_{k}^{\ast}\right)  \in \linebreak M\left(  C\mathbf{x}%
_{k-1},\left(  \mathbf{x}_{k-1}^{\ast}\right)  _{1},z_{k}\right)  $.
\end{proposition}

\begin{proof}
First note that, from the discussion above, there does exist $\left(
y_{1:k-1},v_{1:k-1}\right)  \in\arg\max\mathcal{E}_{z_{1:k-1}}(\mathbf{x}%
_{k-1}^{\ast})$ and $\left(  y_{1:k-1}^{\ast},v_{1:k-1}^{\ast}\right)  \in
\arg\min$$\mathcal{R}_{z_{1:k-1}}(\mathbf{x}_{k-1}^{\ast}),$ and that \linebreak $\left(
y_{1:k-1}^{\ast},v_{1:k-1}^{\ast}\right)  $ is aligned with $\left(
y_{1:k-1},v_{1:k-1}\right)  $.

Suppose $\left(  v_{k},y_{k},v_{k}^{\ast},y_{k}^{\ast}\right)  \in M\left(
C\mathbf{x}_{k-1},\left(  \mathbf{x}_{k-1}^{\ast}\right)  _{1},z_{k}\right)
$. It follows from the state space representation of the estimator system
(\ref{ss2}) that, since $y_{k}$ satisfies $y_{k}-n_{1}v_{k}=C\mathbf{x}_{k-1}$
(that is $y_{k}=C\mathbf{x}_{k-1}+D_{1}v_{k}$), then $\mathbf{x}%
_{k}:=A\mathbf{x}_{k-1}+Bv_{k}$ satisfies $\mathbf{x}_{k}=\mathbf{x}%
_{k}(\mathbf{y},\mathbf{v}),$ where $(\mathbf{y},\mathbf{v})=\left(
y_{1:k},v_{1:k}\right)  $ satisfies the matrix contraint equations for
$\mathcal{E}_{z_{1:k}}(\mathbf{\cdot}).$ Since also $\left\vert v_{k}%
\right\vert \leq1$ and $\left\vert y_{k}-z_{k}\right\vert \leq1$ hold it
follows that $(\mathbf{y},\mathbf{v})$ is feasible for $\mathcal{E}_{z_{1:k}%
}(\mathbf{\cdot})$.

From the regulator system representation (\ref{adjointstate1}),
(\ref{adjointstate}), satisfaction of $v_{k}^{\ast}=C^{\ast}\mathbf{x}%
_{k-1}^{\ast}+D_{1}^{\ast}y_{k}^{\ast}$ by $y_{k}^{\ast}$ and $v_{k}^{\ast}$
implies $(\mathbf{y}^{\ast},\mathbf{v}^{\ast})=\left(  y_{1:k}^{\ast}%
,v_{1:k}^{\ast}\right)  $ is feasible for $\mathcal{R}_{z_{1:k}}%
(\mathbf{x}_{k}^{\ast}),$ where $\mathbf{x}_{k}^{\ast}:=A^{\ast}%
\mathbf{x}_{k-1}^{\ast}+B^{\ast}y_{k}^{\ast}.$ Since $\left(  y_{1:k-1}%
,v_{1:k-1}\right)  $ is aligned with $\left(  y_{1:k-1}^{\ast},v_{1:k-1}%
^{\ast}\right)  ,$ and $\left(  v_{k},y_{k},v_{k}^{\ast},y_{k}^{\ast}\right)
\in M\left(  C\mathbf{x}_{k-1},\left(  \mathbf{x}_{k-1}^{\ast}\right)
_{1},z_{k}\right)  \Rightarrow\left(  y_{k},v_{k}\right)  $ is aligned at time
$k$ with $\left(  y_{k}^{\ast},v_{k}^{\ast}\right)  ,$ we have $(\mathbf{y}%
,\mathbf{v})$ is aligned with $(\mathbf{y}^{\ast},\mathbf{v}^{\ast}).$ We have
shown that $(\mathbf{y},\mathbf{v})$ and $(\mathbf{y}^{\ast},\mathbf{v}^{\ast
})$ are feasible for $\mathcal{E}_{z_{1:k}}(\mathbf{x}_{k}^{\ast})$ and
$\mathcal{R}_{z_{1:k}}(\mathbf{x}_{k}^{\ast})$, and that they are aligned. By
Proposition \ref{primdaulfixed}, $(\mathbf{y},\mathbf{v})$ and $(\mathbf{y}%
^{\ast},\mathbf{v}^{\ast})$ are optimal for $\mathcal{E}_{z_{1:k}}%
(\mathbf{x}_{k}^{\ast})$ and $\mathcal{R}_{z_{1:k}}(\mathbf{x}_{k}^{\ast})$, respectively.
\end{proof}

As an immediate application of Proposition \ref{mainthm} we now prove Theorem
\ref{thmsuccessor}.

\paragraph{Proof of Theorem \ref{thmsuccessor}}

Suppose $\mathbf{x}_{k-1}\in\partial S_{k-1}$, $\mathbf{x}_{k-1}^{\ast}\in
C_{S_{k-1}}^{O}(\mathbf{x}_{k-1})$ and \linebreak $\left(  v_{k},y_{k},v_{k}^{\ast}%
,y_{k}^{\ast}\right)  \in M\left(  C\mathbf{x}_{k-1},\left(  \mathbf{x}%
_{k-1}^{\ast}\right)  _{1},z_{k}\right)  $. Now $\mathbf{x}_{k-1}^{\ast}\in
C_{S_{k-1}}^{O}(\mathbf{x}_{k-1})\Rightarrow\left\langle \mathbf{x}%
_{k-1}^{\ast},\mathbf{x}_{k-1}\right\rangle =h_{S_{k-1}}(\mathbf{x}%
_{k-1}^{\ast})$ so, by Proposition \ref{mainthm}, for any $\left(
y_{1:k-1},v_{1:k-1}\right)  \in\arg\max\mathcal{E}_{z_{1:k-1}}(\mathbf{x}%
_{k-1}^{\ast}),$ and any $\left(  y_{1:k-1}^{\ast},v_{1:k-1}^{\ast}\right)
\in\arg\min$$\mathcal{R}_{z_{1:k-1}}(\mathbf{x}_{k-1}^{\ast}),$ there holds
$\left(  y_{1:k},v_{1:k}\right)  =:(\mathbf{y},\mathbf{v})\in\arg
\max\mathcal{E}_{z_{1:k}}(\mathbf{x}_{k}^{\ast})$ and $\left(  y_{1:k}^{\ast
},v_{1:k}^{\ast}\right)  =:(\mathbf{y}^{\ast},\mathbf{v}^{\ast})\in\arg\min
$$\mathcal{R}_{z_{1:k}}(\mathbf{x}_{k}^{\ast})$, where $\mathbf{x}_{k}^{\ast
}=A^{\ast}\mathbf{x}_{k-1}^{\ast}+B^{\ast}y_{k}^{\ast}.$ Hence, by the second
statement of Proposition \ref{basicprop1}, $\mathbf{x}_{k}%
:=\mathbf{x}_{k}(\mathbf{y},\mathbf{v})\in$ \linebreak $\arg\max\mathcal{E}_{z_{1:k}%
}(\mathbf{x}_{k}^{\ast}),$ and by (\ref{ss2}) $\mathbf{x}_{k}=A\mathbf{x}%
_{k-1}+Bv_{k}$. By assumption $\mathbf{x}_{k}^{\ast}\neq\mathbf{0}$, so
$\mathbf{x}_{k}\in\partial S_{k}$. Then Proposition \ref{propconeargmax}
implies $\mathbf{x}_{k}^{\ast}\in C_{S_{k}}^{O}(\mathbf{x}_{k})$. Finally,
$\mathbf{x}_{k}$ being a successor to $\mathbf{x}_{k-1}$ follows from $\left(
y_{1:k-1},v_{1:k-1}\right)  $ being feasible for $\arg\max\mathcal{E}%
_{z_{1:k-1}}(\mathbf{x}_{k-1}^{\ast})$, the second Condition of Definition
\ref{defM} and Proposition \ref{prop2}. \endproof\

The proof of Theorem \ref{mainthm2} follows directly.

\paragraph{Proof of Theorem \ref{mainthm2}}

If $\mathbf{x}_{k}\in X\left(  \mathbf{x}_{k-1},z_{k}\right)  $ then there
exists $\mathbf{x}_{k-1}^{\ast}\in R(\mathbf{x}_{k-1})$ and $\mathbf{q}\in
M\left(  C\mathbf{x}_{k-1},\left(  \mathbf{x}_{k-1}^{\ast}\right)  _{1}%
,z_{k}\right)  $ such that $\left(  \mathbf{x}_{k},\mathbf{x}_{k}^{\ast
}\right)  :=\left(  A\mathbf{x}_{k-1}+Bq_{1},A^{\ast}\mathbf{x}_{k-1}^{\ast
}+B^{\ast}q_{4}\right)  $ and $\mathbf{x}_{k}^{\ast}\neq\mathbf{0}$\textbf{.}
Since $\mathbf{x}_{k-1}^{\ast}\in C_{S_{k-1}}^{O}(\mathbf{x}_{k-1})$,
by\ Theorem \ref{thmsuccessor}, $\mathbf{x}_{k}\in\partial S_{k}$ and
$\mathbf{x}_{k}$ is a successor to $\mathbf{x}_{k-1}$. \endproof\

From Theorem \ref{mainthm2} we have a procedure that is guaranteed to produce
states that lie on the boundary of $S_{k}$ when $T$ is non-empty. But not yet
addressed is the question: Under what conditions are \textit{all }successors
of $\mathbf{x}_{k-1}\in\partial S_{k-1}$ that lie on the boundary of $S_{k}$
contained in $X(\mathbf{x}_{k-1},z_{k})$? Also, there may be points on the
boundary of $S_{k}$ whose precursors are in the interior of $S_{k-1}$. These
issues are examined next.

\section{Finding precursors of a given $\mathbf{x}_{k}\in\partial S_{k}$}

In the previous Section we were given $\mathbf{x}_{k-1}\in\partial S_{k-1}$
and any $\mathbf{x}_{k-1}^{\ast}\in R(\mathbf{x}_{k-1})$, and showed that all
states $\mathbf{x}_{k}\in X\left(  \mathbf{x}_{k-1},z_{k}\right)  $ belong to
$\partial S_{k}.$ The question is now turned around: For a given
$\mathbf{x}_{k}\in\partial S_{k},$ where are the precursors? To answer this
question we need dynamic programming.

\subsection{Dynamic programming applied to the programs $\mathcal{E}_{z_{1:k}%
}(\mathbf{\cdot})$ and $\mathcal{R}_{z_{1:k}}(\mathbf{\cdot})$}

Let $k>2m$ be an integer. Recall, from (\ref{state_def_star}), $\mathbf{x}%
_{k}^{\ast}\left(  \mathbf{y}^{\ast},\mathbf{v}^{\ast}\right)  :=N_{L}%
^{T}y_{k-m+1:k}^{\ast}+D_{L}^{T}v_{k-m+1:k}^{\ast},$ so $\mathbf{x}%
_{k-1}^{\ast}\left(  \mathbf{y}^{\ast},\mathbf{v}^{\ast}\right)  =N_{L}%
^{T}y_{k-m:k-1}^{\ast}+D_{L}^{T}v_{k-m:k-1}^{\ast}$. Also, from
(\ref{statedef}), $\mathbf{x}_{k-1}\left(  \mathbf{y},\mathbf{v}\right)
=B_{T}^{-1}\left[  -D_{U}y_{k-m:k-1}+N_{U}v_{k-m:k-1}\right]  .$

\begin{proposition}
\label{prop_invimage}For any $\mathbf{x}_{k}^{\ast}\in\mathbb{R}^{m},$ any
$\left(  \mathbf{y}^{\ast},\mathbf{v}^{\ast}\right)  \in\arg\min$%
$\mathcal{R}_{z_{1:k}}(\mathbf{x}_{k}^{\ast})$ and any $\left(  \mathbf{y}%
,\mathbf{v}\right)  \in\arg\max\mathcal{E}_{z_{1:k}}^{\prime}(\mathbf{x}%
_{k}^{\ast}),$ there holds\newline(i) $\left(  y_{1:k-1}^{\ast},v_{1:k-1}%
^{\ast}\right)  \in\arg\min$$\mathcal{R}_{z_{1:k-1}}\left(  \mathbf{x}%
_{k-1}^{\ast}\left(  \mathbf{y}^{\ast},\mathbf{v}^{\ast}\right)  \right)  $,
\newline(ii) $\left(  y_{1:k-1},v_{1:k-1}\right)  \in\arg\max\mathcal{E}%
_{z_{1:k-1}}^{\prime}\left(  \mathbf{x}_{k-1}^{\ast}\left(  \mathbf{y}^{\ast
},\mathbf{v}^{\ast}\right)  \right)  $, and$\newline$(iii) $\left\langle
\mathbf{x}_{k-1}^{\ast}\left(  y^{\ast},v^{\ast}\right)  ,\mathbf{x}%
_{k-1}\left(  y,v\right)  \right\rangle =h_{S_{k-1}}(\mathbf{x}_{k-1}^{\ast
}\left(  \mathbf{y}^{\ast},\mathbf{v}^{\ast}\right)  )$.
\end{proposition}

\begin{proof}
(i) This follows immediately from the dynamic programming principle of
optimality. If $\left(  y_{1:k-1}^{\ast},v_{1:k-1}^{\ast}\right)  \notin%
\arg\min$$\mathcal{R}_{z_{1:k-1}}\left(  \mathbf{x}_{k-1}^{\ast}\left(
\mathbf{y}^{\ast},\mathbf{v}^{\ast}\right)  \right)  $ then, for any \linebreak $\left(
\bar{y}_{1:k-1}^{\ast},\bar{v}_{1:k-1}^{\ast}\right)  \in\arg\min
\mathcal{R}_{z_{1:k-1}}(\mathbf{x}_{k-1}^{\ast}\left(  \mathbf{y}^{\ast
},\mathbf{v}^{\ast}\right)  ),$ we have $\left(  (\bar{y}_{1:k-1}^{\ast}%
,y_{k}^{\ast}),(\bar{v}_{1:k-1}^{\ast},v_{k}^{\ast})\right)  $ is feasible for
$\mathcal{R}_{z_{1:k}}(\mathbf{x}_{k}^{\ast})$ with a lower cost than $\left(
\mathbf{y}^{\ast},\mathbf{v}^{\ast}\right)  \in\arg\min$$\mathcal{R}_{z_{1:k}%
}(\mathbf{x}_{k}^{\ast}),$ a contradiction.\newline(ii) By (i), $\left(
y_{1:k-1}^{\ast},v_{1:k-1}^{\ast}\right)  \in\arg\min$$\mathcal{R}_{z_{1:k-1}%
}\left(  \mathbf{x}_{k-1}^{\ast}\left(  \mathbf{y}^{\ast},\mathbf{v}^{\ast
}\right)  \right)  $, and Theorem \ref{primdaulfixed} applied to
$\mathcal{R}_{z_{1:k}}(\mathbf{x}_{k}^{\ast})$ and $\mathcal{E}_{z_{1:k}%
}^{\prime}(\mathbf{x}_{k}^{\ast})$ implies that $\left(  y_{1:k-1}%
,v_{1:k-1}\right)  $ is aligned with $\left(  y_{1:k-1}^{\ast},v_{1:k-1}%
^{\ast}\right)  $. Since also $\left(  y_{1:k-1},v_{1:k-1}\right)  $ is
feasible for $\mathcal{E}_{z_{1:k-1}}^{\prime}\left(  \mathbf{x}_{k-1}^{\ast
}\left(  \mathbf{y}^{\ast},\mathbf{v}^{\ast}\right)  \right)  $ Theorem
\ref{primdaulfixed} applied to $\mathcal{R}_{z_{1:k-1}}(\mathbf{x}_{k-1}%
^{\ast})$ and $\mathcal{E}_{z_{1:k-1}}^{\prime}(\mathbf{x}_{k-1}^{\ast})$
gives $\left(  y_{1:k-1},v_{1:k-1}\right)  \in\arg\max\mathcal{E}_{z_{1:k-1}%
}^{\prime}\left(  \mathbf{x}_{k-1}^{\ast}\left(  \mathbf{y}^{\ast}%
,\mathbf{v}^{\ast}\right)  \right)  .$ Furthermore, Proposition
\ref{basicprop} applied to $\mathcal{E}_{z_{1:k-1}}\left(  \mathbf{x}%
_{k-1}^{\ast}\left(  \mathbf{y}^{\ast},\mathbf{v}^{\ast}\right)  \right)  $
gives $\mathbf{x}_{k-1}\left(  y,v\right)  \in\arg\max\mathcal{E}_{z_{1:k-1}%
}\left(  \mathbf{x}_{k-1}^{\ast}\left(  \mathbf{y}^{\ast},\mathbf{v}^{\ast
}\right)  \right)  $, implying (iii).
\end{proof}

In Proposition \ref{prop_invimage} a relationship between evolving, connected
estimator and regulator states is given. Some extra notation is helpful in
such situations. In similar fashion to the use of the terms successor and
precursor for estimator states, we make the following definition for regulator
states. Different definitions of the word successor in Definitions
\ref{defprecursor} and \ref{defprecursorreg} should not cause confusion as the
Definition \ref{defprecursor} is used exclusively for unstarred, estimator
variables, and Definition \ref{defprecursorreg} exclusively for starred
regulator variables. Recall from (\ref{state_def_star}) the definition
$\mathbf{x}_{k-1}^{\ast}(\mathbf{y}^{\ast},\mathbf{v}^{\ast}):=N_{L}%
^{T}y_{k-m:k-1}^{\ast}+D_{L}^{T}v_{k-m:k-1}^{\ast}$.

\begin{definition}
\label{defprecursorreg}The vector $\mathbf{x}_{k}^{\ast}$ is a successor to
the vector $\mathbf{x}_{k-1}^{\ast},$ and $\mathbf{x}_{k-1}^{\ast}$ is a
precursor of $\mathbf{x}_{k}^{\ast},$ if there exists $(\mathbf{y}^{\ast
},\mathbf{v}^{\ast})\in\arg\min\mathcal{R}_{z_{1:k}}(\mathbf{x}_{k}^{\ast})$
and $\mathbf{x}_{k-1}^{\ast}=\mathbf{x}_{k-1}^{\ast}(\mathbf{y}^{\ast
},\mathbf{v}^{\ast}).$
\end{definition}

For the case $\mathbf{x}_{k}^{\ast}\neq\mathbf{0}$ Proposition
\ref{prop_invimage} yields the following useful result.

\begin{theorem}
\label{bigcor}Let $\mathbf{x}_{k-1}$ be a precursor of $\mathbf{x}_{k}%
\in\partial S_{k}$, and let $\mathbf{x}_{k-1}^{\ast}$ be a precursor of
$\mathbf{x}_{k}^{\ast}\in C_{S_{k}}^{O}\left(  \mathbf{x}_{k}\right)  $. Then%
\begin{align*}
\text{(i) }\mathbf{x}_{k-1}  &  \in\operatorname*{int}S_{k-1}\Rightarrow
\mathbf{x}_{k-1}^{\ast}=\mathbf{0}\text{; and}\\
\text{(ii) if }\mathbf{x}_{k-1}  &  \in\partial S_{k-1}\text{ and }%
\mathbf{x}_{k-1}^{\ast}\neq\mathbf{0}\text{, then }\mathbf{x}_{k-1}^{\ast}\in
C_{S_{k-1}}^{O}(\mathbf{x}_{k-1})\text{.}%
\end{align*}
\end{theorem}
\begin{proof}
Since $\mathbf{x}_{k}$ is a successor to $\mathbf{x}_{k-1}$, by Proposition
\ref{prop2} there exists $\mathbf{y=}y_{1:k}$ and $\mathbf{v=}v_{1:k}$ such
that $\left(  \mathbf{y},\mathbf{v}\right)  $ is feasible for the program
$\mathcal{E}_{z_{1:k}}^{\prime}(\mathbf{\cdot})$, $\mathbf{x}_{k}%
=\mathbf{x}_{k}\left(  \mathbf{y},\mathbf{v}\right)  $, and $\mathbf{x}%
_{k-1}=\mathbf{x}_{k-1}\left(  \mathbf{y},\mathbf{v}\right)  $. By Proposition
\ref{prop_invimage}, $\mathbf{x}_{k-1}\in\arg\max\mathcal{E}_{z_{1:k}%
}(\mathbf{x}_{k-1}^{\ast})$. Then (i) follows from statement 4 of Proposition
\ref{basicprop}, and (ii) follows from the first statement in Proposition
\ref{basicprop} and Proposition \ref{propconeargmax}.
\end{proof}

\subsection{Precursors of $\mathbf{x}_{k}\in\partial S_{k}$ that lie on the
boundary of $S_{k-1}\label{sect_mainthm}$}

This Section is devoted to a proof of Theorem \ref{thmbig1}, which says that
if a state $\mathbf{x}_{k}$ is in a particular subset of the boundary of
$S_{k}$, and is a successor to some state $\mathbf{x}_{k-1}$ on the boundary
of $S_{k-1}$, then determining $X(\mathbf{x}_{k-1},z_{k})$ suffices to produce
$\mathbf{x}_{k}$. Fortunately this subset of the boundary of $S_{k}$ is big
enough to include all vertices of $S_{k}$.

Some preliminary results are required. The first concerns direction vectors in
$C_{S_{k-1}}^{O}(\mathbf{x}_{k-1})$. A simplifying feature of the results in
Theorems \ref{mainthm2} and \ref{thmbig} is that only one element of the cone
$C_{S_{k-1}}^{O}(\mathbf{x}_{k-1})$ is needed to propagate $\mathbf{x}_{k-1}$
to $\mathbf{x}_{k}$, because the set $X\left(  \mathbf{x}_{k-1},z_{k}\right)
$ is constructed from only one such element. In our proofs it is often
convenient to argue using the set $X^{O}$ defined below; the fact that
Theorems \ref{mainthm2} and \ref{thmbig} can be stated simply in terms of $X$
depends on Proposition \ref{propX0X} below.

\begin{definition}
Given any $\mathbf{x}_{k-1}\in\partial S_{k-1}$, the set $X^{O}\left(
\mathbf{x}_{k-1},z_{k}\right)  =X^{O}$ is defined as%
\[
X^{O}:=%
{\displaystyle\bigcup\limits_{\mathbf{x}_{k-1}^{\ast}\in C_{S_{k-1}}%
^{O}\left(  \mathbf{x}_{k-1}\right)  }}
\left\{
\begin{array}
[c]{c}%
\mathbf{x}_{k}:\mathbf{x}_{k}=A\mathbf{x}_{k-1}+Bq_{1},\mathbf{q}\in M\left(
C\mathbf{x}_{k-1},\left(  \mathbf{x}_{k-1}^{\ast}\right)  _{1},z_{k}\right)
\\
\text{and }A^{\ast}\mathbf{x}_{k-1}^{\ast}+B^{\ast}q_{4}\neq\mathbf{0}%
\end{array}
\right\}  \text{.}%
\]
\end{definition}

From Definition \ref{defti} and Proposition \ref{propTi} we have, for an
arbitrarily selected $\mathbf{x}_{k-1}^{\ast}\in R(\mathbf{x}_{k-1})$, that
$X=X\left(  \mathbf{x}_{k-1},z_{k}\right)  $ is given by
\[
X=\left\{  \mathbf{x}_{k}:\mathbf{x}_{k}=A\mathbf{x}_{k-1}+Bq_{1}%
,\mathbf{q}\in M\left(  C\mathbf{x}_{k-1},\left(  \mathbf{x}_{k-1}^{\ast
}\right)  _{1},z_{k}\right)  ,A^{\ast}\mathbf{x}_{k-1}^{\ast}+B^{\ast}%
q_{4}\neq\mathbf{0}\right\}  \text{.}%
\]
The set $X^{O}$ would appear to be bigger than $X$, so the following
Proposition is at first sight surprising.

\begin{proposition}
\label{propX0X}For any $\mathbf{x}_{k-1}\in\partial S_{k-1}$ there holds
$X^{O}\left(  \mathbf{x}_{k-1},z_{k}\right)  =X\left(  \mathbf{x}_{k-1}%
,z_{k}\right)  $.
\end{proposition}

\begin{proof}
Obviously $X\subseteq X^{O}$, so the proof is complete if it can be shown that
$\mathbf{x}_{k}\in X^{O}\Rightarrow\mathbf{x}_{k}\in X$. We assume
$\mathbf{x}_{k}\in X^{O}$ and, for any $\mathbf{x}_{k-1}^{\ast}\in C_{S_{k-1}%
}^{O}\left(  \mathbf{x}_{k-1}\right)  $, case split the three possibilities
$\mathbf{x}_{k-1}^{\ast}\in R_{i}$. In each case it is shown that
$\mathbf{x}_{k}\in X$.\newline Case (i). If $\mathbf{x}_{k-1}^{\ast}\in R_{1}%
$, then $R\left(  x_{k-1}\right)  =R_{1}\neq\emptyset$, and $\mathbf{x}%
_{k-1}^{\ast}\in R(\mathbf{x}_{k-1})\Rightarrow\mathbf{x}_{k}\in X.$\newline
Case (ii). Now suppose $\mathbf{x}_{k-1}^{\ast}\in R_{2}$. If $R_{1}$ is empty
then $\mathbf{x}_{k-1}^{\ast}\in R_{2}=R(\mathbf{x}_{k-1})\Rightarrow
\mathbf{x}_{k}\in X$. So assume $R=R_{1}\neq\emptyset$. Select any
$\mathbf{\bar{x}}_{k-1}^{\ast}\in R(\mathbf{x}_{k-1})$, so $\left(
\mathbf{\bar{x}}_{k-1}^{\ast}\right)  _{1}=0$. Now%
\begin{align*}
\mathbf{q}  &  \mathbf{=}\left(  q_{1},q_{2},q_{3},q_{4}\right)  \in M\left(
C\mathbf{x}_{k-1},\left(  \mathbf{x}_{k-1}^{\ast}\right)  _{1},z_{k}\right)
\Rightarrow\left(  q_{1},q_{2},0,0\right)  \in M\left(  C\mathbf{x}%
_{k-1},\left(  \mathbf{\bar{x}}_{k-1}^{\ast}\right)  _{1},z_{k}\right) \\
&  \Rightarrow\left(  A\mathbf{x}_{k-1}+Bq_{1},A^{\ast}\mathbf{\bar{x}}%
_{k-1}^{\ast}\right)  \in T\left(  \mathbf{x}_{k-1},\mathbf{\bar{x}}%
_{k-1}^{\ast},z_{k}\right) \\
&  \Rightarrow\mathbf{x}_{k}\in X\left(  T\left(  \mathbf{x}_{k-1}%
,\mathbf{\bar{x}}_{k-1}^{\ast},z_{k}\right)  \right) \\
&  \Rightarrow\mathbf{x}_{k}\in X\left(  \mathbf{x}_{k-1},z_{k}\right)  \text{
by Proposition \ref{propTi}, as required.}%
\end{align*}
Case (iii), that is $\mathbf{x}_{k-1}^{\ast}\in R_{3}$, is similar to case
(ii).\newline It has been shown that, for any $\mathbf{x}_{k-1}^{\ast}\in
C_{S_{k-1}}^{O}(\mathbf{x}_{k-1})$, there holds\newline$\left\{
\mathbf{x}_{k}:\mathbf{x}_{k}=A\mathbf{x}_{k-1}+Bq_{1},\mathbf{q}\in M\left(
C\mathbf{x}_{k-1},\left(  \mathbf{x}_{k-1}^{\ast}\right)  _{1},z_{k}\right)
,A^{\ast}\mathbf{x}_{k-1}^{\ast}+B^{\ast}q_{4}\neq\mathbf{0}\right\}
\subseteq X$, and the result follows.
\end{proof}

Another preparatory result is the following.

\begin{proposition}
\label{prop9}If $\mathbf{x}_{k}\in\partial S_{k}$ and $\mathbf{x}_{k}^{\ast
}\in C_{S_{k}}^{O}(\mathbf{x}_{k})$ then for any $\left(  \mathbf{y}%
,\mathbf{v}\right)  \in\arg\max\mathcal{E}_{z_{1:k}}^{\prime}(\mathbf{x}%
_{k}^{\ast})$ and any $\left(  \mathbf{y}^{\ast},\mathbf{v}^{\ast}\right)
\in\arg\min$$\mathcal{R}_{z_{1:k}}(\mathbf{x}_{k}^{\ast})$ we have\newline%
$\left(  v_{k},y_{k},v_{k}^{\ast},y_{k}^{\ast}\right)  \in M\left(
C\mathbf{x}_{k-1},\left(  \mathbf{x}_{k-1}^{\ast}\right)  _{1},z_{k}\right)
$, where $\mathbf{x}_{k-1}$ is any precursor of $\mathbf{x}_{k}$, and
$\mathbf{x}_{k-1}^{\ast}$ is any precursor of $\mathbf{x}_{k}^{\ast}$.
\end{proposition}

\begin{proof}
The proof is complete if it can be shown that the four conditions in
Definition \ref{defM} are satisfied when $s=C\mathbf{x}_{k-1}$ and $t=\left(
\mathbf{x}_{k-1}^{\ast}\right)  _{1}$. The first condition holds because
$\left(  \mathbf{y},\mathbf{v}\right)  $ is feasible for $\mathcal{E}%
_{z_{1:k}}^{\prime}(\mathbf{x}_{k}^{\ast}).$ Since $\mathbf{x}_{k}$ is a
successor to $\mathbf{x}_{k-1},$ by (\ref{ss2}) we have $y_{k}-n_{1}v_{k}=s,$
and $\mathbf{x}_{k}^{\ast}$ being a successor to $\mathbf{x}_{k-1}^{\ast}$
implies, using (\ref{adjointstate}), that $d_{m+1}v_{k}^{\ast}+n_{m+1}%
y_{k}^{\ast}=-t.$ This verifies the second and third conditions. Finally, by
Proposition \ref{primdaulfixed} applied to $\mathcal{E}_{z_{1:k}}^{\prime
}(\mathbf{x}_{k}^{\ast})$ and $\mathcal{R}_{z_{1:k}}(\mathbf{x}_{k}^{\ast})$
it follows that $\left(  y_{k},v_{k}\right)  $ and $\left(  y_{k}^{\ast}%
,v_{k}^{\ast}\right)  $ are aligned at time $k.$
\end{proof}

Notation for a pair of opposing faces of the polytope $S_{k}$ is required.

\begin{notation}
Suppose $\operatorname*{int}S_{k}\neq\emptyset$. Then $F_{k}^{+}:=H^{+}\cap
S_{k}$, \linebreak $H^{+}=\left\{  \mathbf{x}:\left\langle \mathbf{x},B^{\ast
}\right\rangle =h_{S_{k}}(B^{\ast})\right\}  $ and $F_{k}^{-}:=H^{-}\cap
S_{k}$, $H^{-}=\left\{  \mathbf{x}:\left\langle \mathbf{x},-B^{\ast
}\right\rangle =h_{S_{k}}(-B^{\ast})\right\}  $.
\end{notation}
The following result is intuitively obvious but important, so we provide a proof.

\begin{proposition}
\label{proprelint}Let $\mathbf{x}\in\partial S_{k}$. If $\mathbf{x}^{\ast}\in C_{S_{k}%
}^{O}\left(  \mathbf{x}\right)  $
is unique up to multiplication by a positive scalar,
then $\mathbf{x}\in\operatorname*{relint}F$, where $F=S_{k}\cap H$ is a face
of $S_{k}$ and $H$ is the hyperplane with direction $\mathbf{x}^{\ast}$
supporting $S_{k}$ at $\mathbf{x}$\textbf{.}
\end{proposition}

\begin{proof}
The boundary of $S_{k}$ is given by hyperplanes $H=\left\{  y:\left\langle
\mathbf{x}^{\ast}(i),\mathbf{y}\right\rangle =c_{i}\right\}  $ for
$i=1,\ldots,N$ such that $\mathbf{x}^{\ast}(i)\neq\lambda\mathbf{x}^{\ast}(j)$
for all $\lambda>0$ as long as $i\neq j$. So $S_{k}=\cap_{i=1}^{N}\left\{
y:\left\langle \mathbf{x}^{\ast}(i),\mathbf{y}\right\rangle \leq
c_{i}\right\}  $. Suppose $\mathbf{x}\notin\operatorname*{relint}F$. Then
$\mathbf{x}$ is on, at least, two hyperplanes, $H_{1}$ and $H_{2}$ say; that
is $\left\langle \mathbf{x}^{\ast}(i),\mathbf{y}\right\rangle =c_{i}$,
$i=1,2$. It follows that $\left\langle \mathbf{x}^{\ast}(i),\mathbf{y}%
\right\rangle \leq\left\langle \mathbf{x}^{\ast}(i),\mathbf{x}\right\rangle $
for all $\mathbf{y}\in S_{k}$, that is $\left\langle \mathbf{x}_{i}^{\ast
},\mathbf{x}\right\rangle =h_{S_{k}}(\mathbf{x}_{i}^{\ast})$. By the
uniqueness of $\mathbf{x}^{\ast}$ we have $\mathbf{x}^{\ast}(1)=\mu
\mathbf{x}^{\ast}(2)=\mathbf{x}^{\ast}$ for some $\mu>0$, a contradiction.
\end{proof}

We are finally able to prove Theorem \ref{thmbig1}.

\begin{theorem}
\label{thmbig1}Let $\mathbf{x}_{k-1}\in\partial S_{k-1}$ be given. If
$\mathbf{x}_{k}\in\partial S_{k}\setminus\left(  \operatorname*{relint}%
F_{k}^{+}\cup\operatorname*{relint}F_{k}^{-}\right)  $ and $\mathbf{x}_{k}$ is
a successor to $\mathbf{x}_{k-1}$ then $\mathbf{x}_{k}\in X(\mathbf{x}%
_{k-1},z_{k})$. Furthermore, for all $\mathbf{x}_{k-1}^{\ast}\in R\left(
\mathbf{x}_{k-1}\right)  $, there holds $\left(  \mathbf{x}_{k},\mathbf{x}%
_{k}^{\ast}\right)  \in T\left(  \mathbf{x}_{k-1},\mathbf{x}_{k-1}^{\ast
},z_{k}\right)  $, where $\mathbf{x}_{k}^{\ast}\in C_{S_{k}}^{O}\left(
\mathbf{x}_{k}\right)  $.
\end{theorem}

\begin{proof}
By the contrapositive of Proposition \ref{proprelint}, if $\mathbf{x}_{k}%
\in\partial S_{k}\setminus\left(  \operatorname*{relint}F_{k}^{+}%
\cup\operatorname*{relint}F_{k}^{-}\right)  $ then there exists $\mathbf{x}%
_{k}^{\ast}\in C_{S_{k}}^{O}\left(  \mathbf{x}_{k}\right)  $ where
$\mathbf{x}_{k}^{\ast}\neq\alpha B^{\ast}$ for any scalar $\alpha$. Now any
precursor $\mathbf{x}_{k-1}^{\ast}$ of $\mathbf{x}_{k}^{\ast}$ satisfies
$\mathbf{x}_{k}^{\ast}=A^{\ast}\mathbf{x}_{k-1}^{\ast}+B^{\ast}y^{\ast}$ for
some scalar $y^{\ast}$, so $\mathbf{x}_{k-1}^{\ast}\neq\mathbf{0}$\textbf{.
}By Theorem \ref{bigcor} $\mathbf{x}_{k-1}^{\ast}\in C_{S_{k-1}}^{O}\left(
\mathbf{x}_{k-1}\right)  $. The fact that $\mathbf{x}_{k}$ is a successor to
$\mathbf{x}_{k-1}$ implies $\mathbf{x}_{k}=A\mathbf{x}_{k-1}+Bv$ for some
scalar $v$. By Proposition \ref{prop9}, there holds\newline$\left(
v,C\mathbf{x}_{k-1}+n_{1}v,-\left(  \left(  \mathbf{x}_{k-1}^{\ast}\right)
_{1}+n_{m+1}y^{\ast}\right)  /d_{m+1},y^{\ast}\right)  \in M\left(
C\mathbf{x}_{k-1},\left(  \mathbf{x}_{k-1}^{\ast}\right)  _{1},z_{k}\right)
$, implying, by Theorem \ref{thmsuccessor}, that $\left(  \mathbf{x}%
_{k},\mathbf{x}_{k}^{\ast}\right)  \in T\left(  \mathbf{x}_{k-1}%
,\mathbf{x}_{k-1}^{\ast},z_{k}\right)  $ and $\mathbf{x}_{k}\in X^{O}\left(
\mathbf{x}_{k-1},z_{k}\right)  $. Then $\mathbf{x}_{k}\in X(\mathbf{x}%
_{k-1},z_{k})$ by Proposition \ref{propX0X}.
\end{proof}

\subsection{Precursors of $\mathbf{x}_{k}\in\partial S_{k}$ that lie in the
interior of $S_{k-1}\label{sectinterior}$}

This Section is concerned with propagating the interior of $S_{k-1}$.
Understanding this is necessary in order to identify which states on the
boundary of $S_{k}$ have precursors on the boundary of $S_{k-1}$. Only then
will we be able to guarantee, by using also Theorem \ref{thmbig1}, that all
vertices of $S_{k}$ belong to $X(\mathbf{x}_{k-1},z_{k})$ for some
$\mathbf{x}_{k-1}\in\partial S_{k-1}$.

\begin{theorem}
\label{thm_int}(i) Suppose $\mathbf{x}_{k-1}\in\operatorname*{int}S_{k-1}$. If
$\mathbf{x}_{k}\in\partial S_{k}$ is a successor to $\mathbf{x}_{k-1},$ then
precisely one of $\mathbf{x}_{k}\in\operatorname*{relint}F_{k}^{+}$ or
$\mathbf{x}_{k}\in\operatorname*{relint}F_{k}^{-}$ must hold.\newline(ii) If
$\mathbf{x}_{k}\in\partial S_{k}\setminus\left(  \operatorname*{relint}%
F_{k}^{+}\cup\operatorname*{relint}F_{k}^{-}\right)  $ then all precursors
$\mathbf{x}_{k-1}$ of $\mathbf{x}_{k}$ satisfy $\mathbf{x}_{k-1}\in\partial
S_{k-1}$.
\end{theorem}

\begin{proof}
(i) Suppose $\mathbf{x}_{k-1}\in\operatorname*{int}S_{k-1}$ has a successor
$\mathbf{x}_{k}\in\partial S_{k}.$ For any $\mathbf{x}_{k}^{\ast}\in C_{S_{k}%
}^{O}\left(  \mathbf{x}_{k}\right)  $, and any precursor $\mathbf{x}%
_{k-1}^{\ast}$ of $\mathbf{x}_{k}^{\ast}$, by Theorem \ref{bigcor} we have
$\mathbf{x}_{k-1}^{\ast}=\mathbf{0.}$ Thus all precursors of any
$\mathbf{x}_{k}^{\ast}\in C_{S_{k}}^{O}\left(  \mathbf{x}_{k}\right)  $ are
the zero vector so, by (\ref{adjointstate1}), any $\mathbf{x}_{k}^{\ast}\in
C_{S_{k}}^{O}\left(  \mathbf{x}_{k}\right)  $ must be of the form $\pm\alpha
B^{\ast}$ for some non-zero scalar $\alpha$. This means that $\mathbf{x}_{k}$
must lie either in the face $F_{k}^{+}$, or in the face $F_{k}^{-}$. In fact
either $\mathbf{x}_{k}\in\operatorname*{relint}F_{k}^{+}$ or $\mathbf{x}%
_{k}\in\operatorname*{relint}F_{k}^{-}$ must hold because, up to
multiplication by a positive scalar, $B^{\ast}(-B^{\ast})$ in the definition
of $F_{k}^{+}(F_{k}^{-})$ is unique, and Proposition \ref{proprelint} implies
$\mathbf{x}_{k}\in\operatorname*{relint}F_{k}^{+}\cup\operatorname*{relint}%
F_{k}^{-}$.

To show (ii), assume $\mathbf{x}_{k}\in\partial S_{k}\setminus\left(
\operatorname*{relint}F_{k}^{+}\cup\operatorname*{relint}F_{k}^{-}\right)  .$
By the definitions of $F_{k}^{+}$ and $F_{k}^{-}$, there exists $\mathbf{x}%
_{k}^{\ast}\neq\alpha B^{\ast}$, $\alpha\neq0$, for which $\mathbf{x}%
_{k}^{\ast}\in C_{S_{k}}^{O}(\mathbf{x}_{k})$. For any precursor
$\mathbf{x}_{k-1}^{\ast}$ of $\mathbf{x}_{k}^{\ast}$ there exists $y^{\ast}$
for which $\mathbf{x}_{k}^{\ast}=A^{\ast}\mathbf{x}_{k-1}^{\ast}+B^{\ast
}y^{\ast}$, so $\mathbf{x}_{k-1}^{\ast}\neq\mathbf{0}$\textbf{. }By
Theorem\textbf{ }\ref{bigcor}, for any precursor $\mathbf{x}_{k-1}$ of
$\mathbf{x}_{k},$ we have $\mathbf{x}_{k-1}\in\partial S_{k-1}$.
\end{proof}

Theorem \ref{thm_int} describes all circumstances under which a point in the
interior of $S_{k-1}$ can propagate to a point on the boundary of $S_{k}$. One
interesting corollary follows from the fact that the face $F_{k}^{+}$ (or
$F_{k}^{-}$) will have empty relative interior if and only if it contains a
single point, that point being a vertex of $S_{k}.$ Hence, if $F_{k}^{+}$ and
$F_{k}^{-}$ each contain a single vertex of $S_{k}$, by Theorem \ref{thm_int}
all precursors of all $\mathbf{x}_{k}\in\partial S_{k}$ are in $\partial
S_{k}$.

\section{Vertex results and discussion\label{Sectend}}

By combining previous results the proof of Theorem \ref{thmbig} can now be given.

\paragraph{Proof of Theorem \ref{thmbig}}

Although there may exist $\mathbf{x}_{k-1}\in S_{k-1}$ with no successor, it
is clear from (\ref{wits_recurs}) that every $\mathbf{x}_{k}\in S_{k}$ is a
successor to at least one $\mathbf{x}_{k-1}\in S_{k-1}$. In particular every
vertex of $S_{k}$ has at least one precursor $\mathbf{x}_{k-1}\in S_{k-1}$.
Now all vertices of $S_{k}$ belong to $\partial S_{k}\setminus\left(
\operatorname*{relint}F_{k}^{+}\cup\operatorname*{relint}F_{k}^{-}\right)  $
so, by the second statement of Theorem \ref{thm_int}, any precursor
$\mathbf{x}_{k-1}$ of any vertex of $S_{k}$ satisfies $\mathbf{x}_{k-1}%
\in\partial S_{k-1}$. The Theorem statements then follow from Theorem
\ref{thmbig1}. \endproof\

The ability to propagate
exactly any state on the boundary of $S_{k-1},$ along with the direction of
supporting hyperplanes, is obviously useful. We conclude with some remarks on how the results in this paper might be used
to update $S_{k-1}$ to the whole of $S_{k}$. How best to achieve this in a computationally effective scheme requires further work.

Suppose $\partial S_{k-1}$ is known. By Theorem \ref{thmbig} all vertices of
$S_{k}$ have precursors in $\partial S_{k-1}$. It would be useful to be able
to identify these precursors, so all vertices of $S_{k}$ can be found. Some of
these precursors are themselves vertices of $S_{k-1}$, so it makes sense to
use Theorem \ref{thmbig1} to find all of the successors of vertices of
$S_{k-1}$ that lie in $\partial S_{k}$. But some of the vertices of $S_{k}$
may be successors to states that are not vertices of $S_{k-1}.$ It is believed that the results in this paper will provide the tools needed to locate them.
This is a topic for future research.

Another issue is the propagation of directions of supporting hyperplanes. To
continue the recursion from $S_{k}$ to $S_{k+1}$, for precursors
$\mathbf{\bar{x}}_{k}$ of vertices $\mathbf{x}_{k+1}$ of $S_{k+1}$,\ an
element of each $R(\mathbf{\bar{x}}_{k})$ is needed. In principle this is
known if $S_{k}$ is known, because $S_{k}$ determines all $C_{S_{k}}%
^{O}\left(  \mathbf{\bar{x}}_{k}\right)  $. However, finding even one element
of $R(\mathbf{\bar{x}}_{k})$ knowing only the vertex set of $S_{k}$ is not a
computationally simple task. From the dual recursion we have at least one
element of $C_{S_{k}}^{O}\left(  \mathbf{\bar{x}}_{k}\right)  $. If this
element happens to be in $R(\mathbf{\bar{x}}_{k})$ then $\mathbf{x}_{k+1}$ is
easily found. It is not yet clear how best to proceed if no element of
$R(\mathbf{\bar{x}}_{k})$ is readily available. This also is a topic for
future work.


\Appendix\section*{}

\paragraph{Proof of Proposition \ref{primdaulfixed}}

After expressing the program $\mathcal{R}_{z_{1:k}}(\mathbf{x}^{\ast})$ as an
equivalent linear program, the standard duality result in asymmetric form (\cite{LUEN-84} p. 86, 96) is used:%
\begin{equation}%
\begin{array}
[c]{cc}%
\text{Primal} & \text{Dual}\\%
\begin{array}
[c]{c}%
\min\mathbf{c}^{T}\mathbf{x}\\
\text{s. t. }A\mathbf{x}=\mathbf{b}\\
\mathbf{x}\geq0
\end{array}
&
\begin{array}
[c]{c}%
\max\mathbf{\lambda}^{T}\mathbf{b}\\
\text{s. t. }A^{T}\mathbf{\lambda}\leq\mathbf{c}%
\end{array}
\text{,}%
\end{array}
\label{asymdual}%
\end{equation}
where \textit{complementary slackness} holds: Let $\mathbf{x}$ and
$\mathbf{\lambda}$ be feasible solutions for the primal and dual problems,
respectively. A necessary and sufficient condition that they both be optimal
solutions is that for all $i$

i) $x_{i}>0\Rightarrow a_{i}^{T}\lambda=c_{i}$ (where $a_{i}^{T}$ is the i'th
row of $A^{T}$)

ii) $x_{i}=0\Leftarrow a_{i}^{T}\lambda<c_{i}.$

Note that the use of the symbol $\mathbf{x}$ for the primal decision variable
in (\ref{asymdual}) is different from the use of the symbols $\mathbf{x}_{0}$,
$\mathbf{x}_{0}^{\ast}$, $\mathbf{x}_{k}$, and $\mathbf{x}_{k}^{\ast}$, which
retain their meanings given in the body of the paper.

The program $\mathcal{R}_{z_{1:k}}(\mathbf{x}_{k}^{\ast})$ has a convex
piecewise linear cost function and linear constraints. There is a standard
procedure, which we now follow, for converting such a program to an equivalent
linear programming problem. Introduce new non-negative $k$-dimensional column
vectors $\mathbf{v}^{\ast+},\mathbf{v}^{\ast-},\mathbf{y}^{\ast+}$ and
$\mathbf{y}^{\ast-}$, and put $v_{j}^{\ast}=v_{j}^{\ast+}-v_{j}^{\ast-}$ and
$y_{j}^{\ast}=y_{j}^{\ast+}-y_{j}^{\ast-}.$ At optimality at least one of
$v_{j}^{\ast+},v_{j}^{\ast-}$, and at least one of $y_{j}^{\ast+},y_{j}%
^{\ast-}$, will be zero, so $\left\vert v_{j}^{\ast}\right\vert =v_{j}^{\ast
+}+v_{j}^{\ast-}$ and $\left\vert y_{j}^{\ast}\right\vert =y_{j}^{\ast+}%
+y_{j}^{\ast-}$. Since $\left\langle \mathbf{x}_{0}^{\ast},\mathbf{x}%
_{0}\right\rangle =-\mathbf{x}_{0}^{T}\left[  N_{U}^{T}y_{1:m}^{\ast}%
+D_{U}^{T}v_{1:m}^{\ast}\right]  $, the primal cost function for
$\mathcal{R}_{z_{1:k}}(\mathbf{x}^{\ast})$, namely $\left\Vert y^{\ast
}\right\Vert _{1}+\left\Vert v^{\ast}\right\Vert _{1}+\left\langle
y_{1:k}^{\ast},z_{1:k}\right\rangle +\left\langle \mathbf{x}_{0}^{\ast
},\mathbf{x}_{0}\right\rangle =:J_{pr},$ can be written as
\[
J_{pr}=\left[  \boldsymbol{1}_{4k}+\mathbf{\delta}+\mathbf{\gamma}\right]
\left[
\begin{array}
[c]{c}%
\mathbf{y}^{\ast+}\\
\mathbf{y}^{\ast-}\\
\mathbf{v}^{\ast+}\\
\mathbf{v}^{\ast-}%
\end{array}
\right]
\]
where$\ \boldsymbol{1}_{4k}$ denotes a $4k-$dimensional row vector of ones,
and the row vectors $\mathbf{\delta}$ and $\mathbf{\gamma}$ are defined by%
\begin{align*}
\mathbf{\delta}  &  :=\left[
\begin{array}
[c]{cccccccc}%
-\mathbf{x}_{0}^{T}N_{U}^{T} & \mathbf{0}_{k-m} & \mathbf{x}_{0}^{T}N_{U}^{T}
& \mathbf{0}_{k-m} & -\mathbf{x}_{0}^{T}D_{U}^{T} & \mathbf{0}_{k-m} &
\mathbf{x}_{0}^{T}D_{U}^{T} & \mathbf{0}_{k-m}%
\end{array}
\right] \\
\mathbf{\gamma}  &  :=\left[
\begin{array}
[c]{ccc}%
z_{1:k}^{T} & -z_{1:k}^{T} & \mathbf{0}_{2k}%
\end{array}
\right]  ,
\end{align*}
where $\mathbf{0}_{k-m}$ denotes a $\left(  k-m\right)  $-dimensional row
vector of zeros.

The constraints for the program $\mathcal{R}_{z_{1:k}}(\mathbf{x}^{\ast})$\ in
terms of the new variables are%
\begin{align*}
\left[
\begin{array}
[c]{cccc}%
N_{k}^{T} & -N_{k}^{T} & D_{k}^{T} & -D_{k}^{T}%
\end{array}
\right]  \left[
\begin{array}
[c]{c}%
\mathbf{y}^{\ast+}\\
\mathbf{y}^{\ast-}\\
\mathbf{v}^{\ast+}\\
\mathbf{v}^{\ast-}%
\end{array}
\right]   &  =\left[
\begin{array}
[c]{c}%
0\\
\vdots\\
0\\
\mathbf{x}_{k}^{\ast}%
\end{array}
\right] \\
y_{j}^{\ast+},y_{j}^{\ast-},v_{j}^{\ast+},v_{j}^{\ast-}  &  \geq0.
\end{align*}
The matrices $D_{k}$ $(N_{k})$ are defined in Section \ref{sectnotprelim}, and
$D_{k}^{T}$ $(N_{k}^{T})$ denotes the transpose of $D_{k}$ $(N_{k}).$

In (\ref{asymdual}) put
\begin{align}
A  &  =\left[
\begin{array}
[c]{cccc}%
N_{k}^{T} & -N_{k}^{T} & D_{k}^{T} & -D_{k}^{T}%
\end{array}
\right] \label{notation_dual}\\
\mathbf{x}  &  =\left[
\begin{array}
[c]{cccc}%
\mathbf{y}^{\ast+T} & \mathbf{y}^{\ast-T} & \mathbf{v}^{\ast+T} &
\mathbf{v}^{\ast-T}%
\end{array}
\right]  ^{T},\text{ }c^{T}=\boldsymbol{1}_{4k}+\mathbf{\delta}+\mathbf{\gamma
}\nonumber\\
\mathbf{b}  &  =\left[
\begin{array}
[c]{cccc}%
0 & \ldots & 0 & \mathbf{x}_{k}^{\ast T}%
\end{array}
\right]  ^{T}.\nonumber
\end{align}

Then by (\ref{asymdual}) the dual of a program equivalent to $\mathcal{R}%
_{z_{1:k}}(\mathbf{x}^{\ast})$ is%
\begin{equation}%
\begin{array}
[c]{c}%
\max\limits_{\mathbf{\lambda}\in{\mathbb{R}}^{k}}\left\langle \mathbf{\lambda
}_{k-m+1:k},\mathbf{x}_{k}^{\ast}\right\rangle \\
\left[
\begin{array}
[c]{c}%
N_{k}\\
-N_{k}\\
D_{k}\\
-D_{k}%
\end{array}
\right]  \mathbf{\lambda}\leq\left[  \boldsymbol{1}_{4k}+\mathbf{\delta
}+\mathbf{\gamma}\right]  ^{T}.
\end{array}
\label{eq12}%
\end{equation}
We now show that this program is equivalent to $\mathcal{E}_{z_{1:k}%
}(\mathbf{x}^{\ast}).$

Put%
\begin{equation}
\mathbf{v}:=D_{k}\mathbf{\lambda}+\left[
\begin{array}
[c]{c}%
D_{U}\mathbf{x}_{0}\\
0
\end{array}
\right]  ;\text{ }\mathbf{y}:=N_{k}\mathbf{\lambda}+\left[
\begin{array}
[c]{c}%
N_{U}\mathbf{x}_{0}\\
0
\end{array}
\right]  \label{eq13}%
\end{equation}
so%
\begin{equation}
\left[
\begin{array}
[c]{c}%
N_{k}\\
-N_{k}\\
D_{k}\\
-D_{k}%
\end{array}
\right]  \mathbf{\lambda}-\mathbf{\delta}^{T}=\left[
\begin{array}
[c]{c}%
\mathbf{y}\\
-\mathbf{y}\\
\mathbf{v}\\
-\mathbf{v}%
\end{array}
\right]  . \label{eq14}%
\end{equation}
Then there exists $\mathbf{\lambda}$ satisfying (\ref{eq13}) if and only if
$\mathbf{v}$ and $\mathbf{y}$ satisfy
\begin{equation}
-N_{k}\mathbf{v}+D_{k}\mathbf{y}=\left[
\begin{array}
[c]{c}%
B_{T}\mathbf{x}_{0}\\
0
\end{array}
\right]  . \label{eq11}%
\end{equation}
To see this, observe that the first $m$ rows of the left hand side of
(\ref{eq11}) are \linebreak $-N_{L}\left[  D_{L}\mathbf{\lambda}+D_{U}\mathbf{x}%
_{0}\right]  +D_{L}\left[  N_{L}\mathbf{\lambda}+N_{U}\mathbf{x}_{0}\right]
=\left[  -N_{L}D_{U}+D_{L}N_{U}\right]  \mathbf{x}_{0}=B_{T}\mathbf{x}_{0}%
,\ $and the other rows of (\ref{eq11}) follow from (\ref{def:BTmatrix}).

Next we show $\left\langle \mathbf{\lambda}_{k-m+1:k},\mathbf{x}^{\ast
}\right\rangle =\left\langle \mathbf{x}_{k}\left(  \mathbf{y},\mathbf{v}%
\right)  ,\mathbf{x}^{\ast}\right\rangle .$ This is true because
\begin{align*}
\mathbf{x}_{k}\left(  \mathbf{y},\mathbf{v}\right)   &  =\left(  B_{T}\right)
^{-1}\left[  N_{U}v_{k-m+1:k}-D_{U}y_{k-m+1:k}\right] \\
&  =\left(  B_{T}\right)  ^{-1}\left[  N_{U}D_{L}\lambda_{k-m+1:k}-D_{U}%
N_{L}\lambda_{k-m+1:k}\right]  \text{ by (\ref{eq13})}\\
&  =\lambda_{k-m+1:k}\text{ by (\ref{def:BTmatrix}).}%
\end{align*}
It remains only to show that the alignment conditions of the Theorem statement
hold. The inequalities $\left\vert v_{j}\right\vert \leq1$ and $\left\vert
y_{j}-z_{j}\right\vert \leq1$ follow from (\ref{eq14}) and the inequalities
(\ref{eq12}). The other inequalities in Definition \ref{def_align} follow
directly from the complementary slackness conditions for (\ref{asymdual}) when
the associations (\ref{notation_dual}) are made.

\bibliographystyle{siam}
\def\cprime{$'$} \def\cprime{$'$} \def\cprime{$'$}

\end{document}